\documentclass[11pt]{amsart}
\usepackage{latexsym}
\usepackage{hyperref} %
\usepackage{amsmath,amsfonts,amssymb}
\numberwithin{equation}{section}
\makeindex
\newtheorem{theo}{Theorem}[section]
\newtheorem{lemma}[theo]{Lemma}
\newtheorem{corol}[theo]{Corollary}
\newtheorem{prop}[theo]{Proposition}
\theoremstyle{definition}
\newtheorem{remark}[theo]{Remark}

\newtheorem{example}[theo]{Example}

\newtheorem{note}[theo]{Note}
\newcommand{\inter}{\operatorname{int}}
\newcommand{\ainter}{\operatorname{aint}}

\newcommand{\aclos}{\operatorname{acl}}

\begin{document}
\title{Normal cones and Thompson metric}
\author{S. Cobza\c s and M.-D. Rus}
\address{S. Cobza\c s, Babes-Bolyai University, Department of Mathematics, Cluj-Napoca, Romania}
\email{scobzas@math.ubbcluj.ro}
\address{M.-D. Rus, Technical University of Cluj-Napoca, Cluj-Napoca, Romania}
\email{rus.mircea@math.utcluj.ro}
\date{\today}
\begin{abstract}
 The aim of this paper is to  study the basic properties of the  Thompson metric $d_T$ in the general case of a linear spaces $X$ ordered by a cone $K$.   We show that $d_T$ has monotonicity properties which make it compatible with the linear structure. We also prove several convexity properties of  $d_T,$  and some results
concerning the topology of  $d_T,$  including a brief study of the  $d_T$-convergence of monotone sequences. It is shown  most   results are true without    any assumption of  Archimedean-type property for $K$.
 One considers various completeness properties and one studies the relations between them.  Since  $d_T$  is defined in the context of a generic ordered linear space, with no need of an underlying
topological structure, one expects to express its completeness in terms of properties of the ordering,
with respect to the linear structure. This is done in this paper and, to the best of our knowledge, this has not been done yet.
  Thompson  metric $d_T$ and order-unit (semi)norms $|\cdot|_u$ are strongly related
and share important properties, as both are defined in terms of the ordered linear structure. Although
 $d_T$  and $|\cdot|_u$   are only topological (and not metrical) equivalent on $K_u$, we   prove that the completeness is a common feature.   One proves  the completeness of the Thompson metric on a sequentially complete normal cone in a locally convex space. At the end of the paper, it is shown that, in the case of a Banach space, the normality of the cone is also necessary for the completeness of the Thompson metric.
\end{abstract}
\maketitle

\tableofcontents

\section{Introduction}
In his study on the foundation of geometry, Hilbert \cite{hilbert1895} introduced a metric in the  Euclidean space, known  now  as the Hilbert projective metric. Birkhoff \cite{birkhoff57} realized that fixed point techniques for nonexpansive mappings with respect to the Hilbert projective metric can be applied to prove the Perron-Frobenius theorem on the existence of eigenvalues and eigenvectors of non-negative square matrices and  of solutions to some integral equations with positive kernel. The result on the Perron-Frobenius theorem was also found independently by Samelson \cite{samelson57}. Birkhoff's proof relied on some results from differential projective geometry, but Bushell \cite{bushell73a,bushell73b} gave new and more accessible proofs to these results by using the Hilbert metric defined on cones, revitalizing the interest for this topic (for a recent account on Birkhoff's definition of the  Hilbert metric see the paper \cite{lem-nuss13},  and for Perron-Frobenius theory, the book \cite{Lem-Nuss12}). A related partial metric on cones in Banach spaces was devised by Thompson \cite{thomp63}, who proved the completeness of this metric (under the hypothesis of the normality of the cone),  as well as some fixed point theorems for contractions with respect to it. It turned out that both these metrics are very useful in a variety of problems in various domains of mathematics and in applications to economy and other fields. Among these applications we mention those  to fixed points for mixed monotone  operators and other classes of operators on ordered vector spaces, see \cite{chen93,chen99,chen01a,chen02,RusM10,rusm11}. Nussbaum alone, or in collaboration with other mathematicians,  studied the limit   sets of iterates of  nonexpansive mappings with respect to Hilbert or Thompson metrics, the  analog of Denjoy-Wolff theorem for iterates of holomorphic mappings, see  \cite{lins07,lins-nuss06,lins-nuss08,Nuss88,Nuss89,nuss07}. These metrics have also interesting applications to operator theory---to  means for positive operators, \cite{iz-nakamura09,nakamura09}, and to isometries in spaces of operators on Hilbert space and in $C^*$-algebras, see \cite{hat-molnar14,molnar09}, and the papers quoted therein.

Good presentations of Hilbert and Thompson metrics are given in the monographs \cite{Hy-Is-Ras,Nuss88,Nuss89}, and in the   papers \cite{akian-nuss13,lem-nuss13,nuss-walsh04}. A more general approach---Hilbert and Thompson metrics on convex sets--- is proposed in the papers  \cite{bauer-bear69} and \cite{bear-weis67}.

The aim of this paper, essentially based on the Ph. D. thesis \cite{RusM10}, is to  study the basic properties of the  Thompson metric $d_T$ in the general case of a vector space $X$ ordered by a cone $K$. Since  $d_T$  is defined in the context of a generic ordered vector space, with no need of an underlying
topological structure, one expects to express its completeness in terms of properties of the ordering,
with respect to the linear structure. This is done in the present paper and, to the best of our knowledge, this has not been done yet.

 For the convenience of the reader,   we survey in Section \ref{S.ord-vs} some notions and  notations   which will be used throughout and list,
without proofs, the most important results that are  assumed to be known. Since there is no a standard terminology in the theory of ordered vector spaces, the main purpose of this
preliminary section is to provide a central point of reference for a unitary treatment of all of the topics
in the rest of the paper. As possible we have given exact references to textbooks were these results can be found, \cite{Alipr,BreckW,Deimling,Guo,Jameson, Peressini,Schaef}.

Section \ref{S.Thompson-metric} is devoted to the definition and basic properties of the Thompson metric.
We show that $d_T$ has monotonicity properties which make it compatible with the linear structure. We also prove several convexity properties of  $d_T.$  We close this section with some results
concerning the topology of  $d_T,$  including a brief study of the  $d_T$-convergence of monotone sequences.
Note that most of these   results are true without the assumption of an Archimedean--type
property for $K$.

 We show that the Thompson metric $d_T$ and order--unit (semi)norms $|\cdot|_u$ are strongly related
and share important properties (e.g., they are topologically equivalent), as both are defined in terms of the ordered linear structure.

Section \ref{S.Completeness} is devoted to  various kinds of completeness. It is shown that,
although $d_T$  and $|\cdot|_u$   are only topologically, and  not metrically, equivalent, we are able to prove that the completeness is a common feature. Also we study a special notion, called self--completeness,  we prove that several completeness  conditions  are equivalent and that the Thompson metric on a sequentially complete normal cone $K$ in a locally convex space $X$ is complete.

  In the last subsection  we show that in the case when $X$ is a Banach space,   the completeness of $K$ with respect to $d_T$ is also necessary for the normality of $K$. This is obtained  as a consequence of a more general result (Theorem \ref{t2.complete-T-metric-B-sp}) on the equivalence of several conditions to the completeness of $K$ with respect to $d_T$ .

\section{Cones in vector spaces}\label{S.ord-vs}

\subsection{Ordered vector spaces}

A \emph{preorder} on a set $Z$ is a reflexive and transitive relation $\,\le\,$  on $Z$. If the relation $\,\le\,$ is also antisymmetric then it is called an \emph{order} on $Z$. If any two elements in $Z$ are comparable (i.e. at least one of the relations $x\le y$ or $y\le x$ holds), then one says that the order (or the preorder) $\le $ is \emph{total}.

  Since in what follows we  shall be concerned only with real vector spaces,  by a ``vector space" we will understand always a ``real vector space".

A nonempty subset $W$ of a     vector space $X$ is called a \emph{wedge} if
\begin{equation}\label{def.cone}
\begin{aligned}
 &{\rm (C1)}\qquad  W+W\subset W, \\
 &{\rm (C2)}\qquad tW\subset W, \quad\mbox{for all } \; t\ge 0.
\end{aligned}\end{equation}

The wedge $W$ induces a preorder on $X$ given by
\begin{equation}\label{def.order}
x\le_W y \iff y-x\in W.\end{equation}

The notation $\,x<_W y\,$ means that $\,x\le_W y\,$ and $\,x\ne y.$  If there is no danger of confusion the subscripts   will be omitted.

This preorder is compatible with the linear structure of $X$, that is
\begin{equation}\label{eq.lin.ord}
\begin{aligned}
 &{\rm (i)}\qquad x\le y\Longrightarrow x+z\le y+z,\quad\mbox{and}  \\
 &{\rm (ii)}\qquad  x\le y\Longrightarrow tx\le ty,
\end{aligned}
\end{equation}
for all $x,y,z\in X$ and $t\in\mathbb{R}_+,$ where $\mathbb{R}_+=\{t\in \mathbb{R} : t\ge 0\}.$  This means that one can add inequalities
$$
x\le y\;\mbox{and}\; x'\le y'\Longrightarrow x+x'\le y+y',$$
and multiply by positive numbers
$$
x\le y \iff tx\le ty,$$
for all $x,x',y,y'\in X$ and $t>0.$  The multiplication by negative numbers reverses the inequalities
$$
\forall t<0,\;\; (x\le y \iff tx\ge ty).$$

As a consequence of this equivalence, a subset $A$ of $X$ is bounded above iff the set $-A$ is bounded below. Also
$$
\inf A =-\sup(-A)\quad \sup A=-\inf(-A).$$

It is obvious that  the preorder $\,\le_W$ is total iff $\,X=W\cup(-W).$

\begin{remark}
It follows that in definitions (or hypotheses) we can ask only one order condition. For instance, if we ask that every bounded above subset of an ordered vector space has a supremum, then every bounded below subset will have an infimum, and consequently, every bounded subset has an infimum and a supremum. Similarly, if a linear preorder is upward directed, then it is automatically downward directed, too.
\end{remark}

Obviously, the wedge $W$ agrees with the set of positive elements in $X$,
$$
W=X_+:=\{x\in X : 0\le_W x\}.$$

Conversely, if $\,\le\,$ is a preorder on a vector space $X$ satisfying \eqref{eq.lin.ord} (such a preorder is called a \emph{linear preorder}), then $W=X_+$ is a wedge in $X$ and $\,\le\, = \,\le_W.$ Consequently, there is a perfect correspondence between linear preorders on a vector space $X$ and wedges in $X$ and so any property in an ordered vector space can be formulated in terms of  the preorder or of the wedge.

A  \emph{cone} $K$ is a wedge satisfying the condition
\begin{equation}\label{def.pcone}
{\rm (C3)}\qquad K\cap(-K)=\{0\}.\end{equation}

This is equivalent to the fact that the   induced preorder is antisymmetric,
\begin{equation}\label{eq.antiref}
x\le y\;\mbox{and}\; y\le x \Longrightarrow y=x,
\end{equation}
for all $x,y\in X,$ that  it is an order on $X$.

A pair $(X,K),$ where $K$ is a cone (or a wedge) in a vector space $X,$ is called an \emph{ordered} (resp. \emph{preodered})  \emph{vector space}.

An \emph{order interval} in an ordered vector space $(X,K)$ is a (possibly empty) set of the form
$$
[x;y]_o=\{z\in X : x\le z\le y\}=(x+K)\cap(y-K),$$
for some $x,y\in X.$ It is clear that an order interval $[x;y]_o$ is a convex subset of $X$ and that
$$
[x;y]_o=x+[0;y-x]_o. $$

The notation $[x;y]$ will be reserved to algebraic intervals: $[x;y]:=\{(1-t) x+t y : t\in [0;1]\}.$

A subset $A$ of $X$ is called \emph{order--convex} (or \emph{full}, or \emph{saturated}) if $[x;y]_o\subset A$ for every $x,y\in A.$ Since the intersection of an arbitrary  family of order--convex sets is order--convex, we can define the order--convex hull $[A]$ of a nonempty subset $A$ of $X$ as the intersection of all order--convex subsets of $X$ containing $A$, i.e. the smallest order--convex subset of $X$ containing $A$. It follows that
\begin{equation}\label{eq.o-cv-hull}
[A]=\bigcup\{[x;y]_o : x,y\in A\} =(A+K)\cap (A-K).
\end{equation}

Obviously, $A$ is order--convex iff $A=[A].$

\begin{remark} It is obvious that if $x\le y,$ then $[x;y]\subset [x;y]_o,$ but the reverse inclusion could not hold as the following example shows.
    Taking $X=\mathbb{R}^2$ with the coordinate order and $x=(0,0),\, y=(1,1),$ then $[x;y]_o$ equals the (full) square with the vertices  $(0,0), \,(0,1),\, (1,1)$ and $(0,1),$ so it is larger  than the segment $[x;y].$
  \end{remark}

We mention also the following result.
\begin{prop}[\cite{BreckW}]\label{p.char-total-o-cv}
Let $(X,\le)$ be an ordered vector space.
Then the order $\,\le\,$ is total iff every order--convex subset of $X$ is convex.
\end{prop}

We shall consider now some algebraic-topological notions concerning the subsets of a vector space $X$. Let $A$ be a subset of $X$.

The subset $A$ is called:
\vspace{2mm}

\textbullet\; \emph{balanced} if $\lambda A\subset A$ for every $|\lambda|\le 1;$

\textbullet\; \emph{symmetric} if $-A=A;$

\textbullet\; \emph{absolutely convex} if it is convex and balanced;

\textbullet\; \emph{absorbing} if $\{t>0 : x\in tA\}\ne \emptyset$ for every $x\in X.$
\vspace{2mm}

The following equivalences are immediate:
\begin{align*}
A \;\mbox{is absolutely convex}\; \iff&\; \forall  a,b \in A,\;\forall  \alpha,\beta \in \mathbb{R},\; \mbox{with}\; |\alpha|+|\beta|=1,\quad \alpha a+\beta b\in A\\
\iff&\; \forall  a,b \in A,\; \forall  \alpha,\beta \in \mathbb{R},\;\mbox{with}\; |\alpha|+|\beta|\le 1,\quad \alpha a+\beta b\in A.
  \end{align*}

 Notice that a balanced set is symmetric and a symmetric convex set containing 0 is balanced.

  The following properties are easily seen.
  \begin{prop}\label{p.full-hull}
    Let $X$ be an ordered vector space and $A\subset X$ nonempty. Then
\begin{enumerate}
  \item[{\rm 1.}] If $A$ is convex, then $[A]$ is also convex.
\item[{\rm 2.}] If $A$ is balanced, then $[A]$ is also balanced.
\item[{\rm 3.}] If $A$ is absolutely convex, then $[A]$ is also absolutely convex.
\end{enumerate}    \end{prop}

One says that $a$ is an \emph{algebraic interior} point of $A$ if
\begin{equation}\label{def.a-int-pt}
\forall x\in X,\; \exists \delta >0,\; \mbox{such that}\; \forall \lambda\in [-\delta;\delta],\; a+\lambda x\in A.
\end{equation}

The (possibly empty) set of all interior points of $A$, denoted by $\ainter(A),$   is called the \emph{algebraic interior} (or the \emph{core}) of the set $A$. It is obvious that if $X$ is a TVS, then  $\inter(A)\subset\ainter(A).$
where  $\inter(A)$ denotes the interior of the set $A$. In finite dimension we have equality, but  the inclusion can be proper if $X$ is infinite dimensional.

A cone $K$ is called \emph{solid} if $\inter(K)\ne\emptyset.$

\begin{remark}
  Z\u alinescu \cite{Zali} uses the notation $A^i$ for the algebraic interior and $^iA$ for the algebraic interior of $A$ with respect to its affine hull (called the relative algebraic interior). In his definition of an algebraic interior point of $A$ one asks that the conclusion of \eqref{def.a-int-pt}   holds only for $\lambda \in [0;\delta],$ a condition equivalent to \eqref{def.a-int-pt}.
\end{remark}

The set $A$ is called \emph{lineally open} (or \emph{algebraically open}) if $A=\ainter (A),$ and \emph{lineally closed} if $X\setminus A$
is lineally open. This is equivalent to the fact that any line in $X$ meets $A$ in a closed subset of the line.
The smallest lineally closed set containing a set $A$ is called the \emph{lineal} (or \emph{algebraic}) closure of $A$ and it is denoted by $\aclos (A).$ Again, if $X$ is a TVS, then any closed subset of $X$ is lineally closed. The subset $A$ is called \emph{lineally bounded} if the intersection with any line $D$ in $X$ is a bounded subset of $D$.
\begin{remark}
 The terms ``lineally open", ``lineally closed", etc, are taken from Jameson \cite{Jameson}.
\end{remark}
\begin{remark}\label{re.a-int}
Similar to the topological case one can prove that
\begin{equation}\label{eq1.re.a-int}
a\in \ainter(A),\; b\in A \;\mbox{and}\;  \lambda\in [0;1)\;\Longrightarrow \; (1-\lambda) a+\lambda b\in \ainter(A).\end{equation}

Consequently, if $A$ is convex then $\ainter(A)$ is also convex.

If $K$ is a cone, then $\ainter(K)\cup\{0\}$ is also a cone and
\begin{equation}\label{eq2.re.a-int}
\ainter(K)+K\subset \ainter(K).\end{equation}
  \end{remark}

  We justify only the second assertion. Let $x\in \ainter(K)$ and $y\in K.$
  Then
  $$
  x+y=2\left(\frac12 x+\frac12 y\right)\in \ainter(K).$$

Now we shall consider some further properties of linear orders. A linear order    $\,\le\,$  on a vector space $X$ is called:

\vspace{2mm}
\textbullet\quad \emph{Archimedean} if for every $x,y\in X,$
\begin{equation}\label{def.Arch}
(\forall n\in\mathbb{N},\; nx\le y)\;\Longrightarrow \; x\le 0;\end{equation}

\textbullet  \quad \emph{almost Archimedean} if for every $x,y\in X,$
\begin{equation}\label{def.a-Arch}
(\forall n\in\mathbb{N},\; -y\le nx\le y)\;\Longrightarrow \; x = 0;\end{equation}

The following four propositions  are taken from Breckner \cite{BreckW} and Jameson \cite{Jameson}. In all of them   $X$ will be a   vector space and $\,\le\,$ a linear preorder on $X$ given by the wedge $W=X_+.$

\begin{prop}\label{p.char.Arch}
  The following are equivalent.
\begin{enumerate}
 \item[{\rm 1.}]  The preorder $\le $ is Archimedean.
 \item[{\rm 2.}] The wedge $W$ is lineally closed.
 \item[{\rm 3.}] For every $x\in X$ and $\,y\in W,\; 0=\inf\{n^{-1}x : n\in\mathbb{N}\}.$
 \item[{\rm 4.}] For every $\,x\in X$ and $y\in W,\; nx\le y,\,$ for all $\, n\in\mathbb{N},\,$ implies $\,x\le 0.$
 \item[{\rm 5.}] For every $A\subset \mathbb{R}$ and $x,y\in X,\;\; y\le \lambda x$ for all $\lambda\in A,$ implies $y\le \mu x,$ where $\mu=\inf A.$
 \end{enumerate}
 \end{prop}

\begin{prop}\label{p.char.a-Arch}
The following are equivalent.
\begin{enumerate}
  \item[{\rm 1.}] The preorder is almost-Archimedean.
  \item[{\rm 2.}]  $\aclos(W)\,$  is a wedge.
  \item[{\rm 3.}] Every order interval in $X$ is lineally bounded.
\end{enumerate}\end{prop}

A wedge $W$ in $X$ is called \emph{generating} if $X=W-W.$ The preorder $\,\le\,$ is called \emph{upward} (\emph{downward}) \emph{directed} if for every $x,y\in X$ there is $z\in X$ such that $x\le z,\, y\le z$ (respectively,   $x\ge z,\, y\ge z$). If the order is linear, then these two notions are equivalent, so we can say simply that $\,\le\,$ is directed.

\begin{prop}\label{p.char.gener-cone}
The following are equivalent.
\begin{enumerate}
  \item[{\rm 1.}] The wedge $W$ is generating.
  \item[{\rm 2.}]  The order $\,\le\,$ is directed.
  \item[{\rm 3.}]  $\forall x\in X,\; \exists y\in W,\; x\le y. $
\end{enumerate}\end{prop}

 Let $(X,W)$ be a preordered vector space. An element $u\in W$ is called an \emph{order unit} if  the set $[-u;u]_o$ is absorbing. It is obvious that an order unit must be different of 0 (provided $X\ne\{0\}$).

\begin{prop}\label{p.char-o-unit}
Let   $u\in W\setminus \{0\}.$ The following are equivalent.
\begin{enumerate}
\item[{\rm 1.}] The element $u$ is an order unit.
\item[{\rm 2.}] The order interval $[0;u]_o$ is absorbing.
\item[{\rm 3.}] The element $u$ belongs to the algebraic interior of $W.$
\item[{\rm 4.}]   $[\mathbb{R} u]=X.$
\end{enumerate}\end{prop}

\subsection{Completeness in ordered vector spaces}
An ordered vector space $X$ is called a vector lattice if  any two elements $x,y\in X$ have a supremum, denoted by
$x\vee y.$ It follows that they have also an infimum, denoted by $x\wedge y,$ and these properties extend to any finite subset of $X$.  The ordered vector space $X$ is called   \emph{order complete} (or \emph{Dedekind complete}) if every bounded from above   subset of $X$ has a supremum and    \emph{order} $\sigma$-\emph{complete} (or \emph{Dedekind $\sigma$-complete}) if  every bounded from above countable   subset of $X$ has a supremum. The fact that every bounded above subset of $X$ has a supremum is equivalent to the fact that every bounded below subset of $X$ has an infimum. Indeed, if $A$ is bounded above, then $\sup\{y : y\; \mbox{is a lower bound for }\; A\} =\inf A.$

\begin{remark}
  An ordered vector space $X$ is order complete iff for each pair $A,B$ of nonempty subsets of $X$ such that $A\le B$ there exists $z\in X$ with $A\le z\le B$.

  This similarity with ``Dedekind cuts" in $\mathbb{R}$ justifies the term \emph{Dedekind complete} used by some authors.
  Here $A\le B $ means that $\, a\le b \,$ for all $(a,b)\in A\times B.$
\end{remark}

The following results gives characterizations of these properties in terms of directed subsets.
\begin{prop}[\cite{Alipr}, Theorem 1.20]\label{p.Dedekind}
Let $X$ be a  vector lattice.
 \begin{enumerate}
 \item[{\rm 1.}] The space $X$ is order complete iff    every   upward directed bounded above subset of  $X$ has a supremum (equivalently, if every bounded above monotone net has a supremum).
 \item[{\rm 2.}] The space $X$ is Dedekind $\sigma$-complete iff every  upward directed bounded above countable subset of $X$ has a supremum (equivalently, if every bounded above monotone sequence has a supremum).
 \end{enumerate}\end{prop}

\subsection{Ordered topological vector spaces (TVS)}

In the case of an ordered TVS $(X,\tau)$ some connections between order and topology hold. In the following propositions $(X,\tau)$ will be a TVS with a preorder or an  order, $\,\le\,$ generated by a wedge $W,$ or  by a cone $K,$ respectively.
We start by a simple result.
\begin{prop}\label{p1.order-tvs}
A wedge  $W$ is closed iff the   inequalities are preserved by limits, meaning that  for all nets $(x_i : i\in I),\, (y_i : i\in I) $ in $X,$
$$
\forall i\in I,\; x_i\le y_i \;\;\mbox{and}\;\; \lim_ix_i=x,\, \lim_iy_i=y \Longrightarrow \; x\le y.$$
\end{prop}

Other results are contained in the following proposition.

\begin{prop}[\cite{Alipr}, Lemmas 2.3 and 2.4]\label{p2.order-tvs}
  Let $(X,\tau)$  be a TVS ordered by a $\tau$-closed cone $K$. Then
  \begin{enumerate}
\item[{\rm 1.}]   The topology $\tau$ is Hausdorff.
\item[{\rm 2.}]  The cone $K$ is Archimedean.
\item[{\rm 3.}]  The order intervals are $\tau$-closed.
\item[{\rm 4.}]  If $(x_i:i\in I)$ is an increasing net which is $\tau$-convergent to $x\in X$, then $x=\sup_ix_i.$
\item[{\rm 5.}] Conversely, if the topology $\tau$ is Hausdorff, $\;\inter(K)\ne\emptyset$ and $K$ is Archimedean, then $K$ is $\tau$-closed.
\end{enumerate}
\end{prop}

\begin{note}
 In what follows by an ordered TVS we shall understand a TVS ordered by a closed cone.
  Also, in an ordered TVS $(X,\tau,K)$ we have some parallel notions---with respect to topology and with respect to order. To make distinction between  them, those referring to order will have the prefix ``order--", as, for instance, ``order--bounded", ``order--complete", etc, while for those referring to topology we shall use the prefix ``$\tau$-", or ``topologically--",
 e.g., ``$\tau$-bounded", ``$\tau$-complete" (resp. ``topologically--bounded", ``topologically--complete"), etc.
\end{note}

\subsection{Normal cones in TVS and in LCS (locally convex spaces)}
Now we  introduce a very important notion in the theory of ordered vector spaces. A cone $K$ in a TVS $(X,\tau)$ is called \emph{normal} if there exists a neighborhood basis at 0 formed of order--convex sets.

The following characterizations are taken from \cite{BreckW} and  \cite{Peressini}.
\begin{theo}\label{t1.char-normal-cone}
  Let $(X,\tau,K)$ be an ordered TVS. The following are equivalent.
  \begin{enumerate}
\item[{\rm 1.}] The cone $K$ is normal.
\item[{\rm 2.}] There exists a basis $\mathcal B$ formed of order--convex balanced  0-neighborhoods.
\item[{\rm 3.}] There exists a basis $\mathcal B$ formed of balanced  0-neighborhoods such that for every $B\in \mathcal B,\; y\in B $ and $0\le x\le y$ implies $x\in B.$
\item[{\rm 4.}] There exists a basis $\mathcal B$ formed of  balanced  0-neighborhoods such that for every $B\in \mathcal B,\; y\in B $  implies $[0;y]_o\subset  B.$
\item[{\rm 5.}] There exists a basis $\mathcal B$ formed of balanced  0-neighborhoods and a number $\gamma >0$  such that for every $B\in \mathcal B,\;   [B]\subset \gamma B.$
\item[{\rm 6.}] If $(x_i:i\in I)$ and   $(y_i:i\in I)$  are two nets in $X$ such that $\forall i\in I,\; 0\le x_i\le y_i$ and $\lim_iy_i=0,$ then    $\lim_ix_i=0.$\end{enumerate}

If further, $X$ is a LCS, then the fact that the cone $K$ is normal is equivalent to each of the conditions 2--5, where the term ``balanced" is replaced with ``absolutely convex".\end{theo}

\begin{remark}
  Condition 6 can be replaced with the equivalent one:

  If $(x_i:i\in I), \,(y_i:i\in I)$ and  $(z_i:i\in I) $ are  nets in $X$ such that $\forall i\in I,\; x_i\le z_i\le y_i$ and $\lim_ix_i=x=\lim_iy_i,$ then $\lim_iz_i=x.$
 \end{remark}

  The normality implies the fact that  the order--bounded sets are bounded.
 \begin{prop}[\cite{Peressini}, Proposition 1.4]\label{p1.normal-cone-bd} If $(X,\tau)$ is a TVS ordered by a normal cone, then every order--bounded subset of $X$ is $\tau$-bounded.
 \end{prop}

 \begin{remark} In the case of a normed space this condition characterizes the normality, see Theorem \ref{t4.char-normal-cone} below. Also, it is clear that a subset $Z$ of an ordered vector space $X$ is   order--bounded iff there exist $x,y\in X$ such that  $Z\subset [x;y]_o.$
   \end{remark}

 The existence of a normal solid cone in a TVS makes the topology normable.
 \begin{prop}[\cite{Alipr}, p. 81, Exercise 11, and \cite{Peressini}]\label{p1.normal-c-TVS}
   If a Hausdorff TVS $(X,\tau)$ contains a solid $\tau$-normal cone, then the topology $\tau$ is normable.
 \end{prop}

 In order to give characterizations of normal cones in LCS we consider some properties of seminorms. Let $\gamma >0$.
 A seminorm $p$ on a vector space $X$ is called:
 \vspace{2mm}

 \textbullet\; $\gamma$-\emph{monotone} if $0\le x\le y\;\Longrightarrow\; p(x)\le \gamma p(y);$

 \textbullet\; $\gamma$-\emph{absolutely monotone} if $-y\le x\le y\;\Longrightarrow\; p(x)\le \gamma p(y);$

 \textbullet\; $\gamma$-\emph{normal} if $x\le z\le y\;\Longrightarrow\; p(z)\le \gamma\max\{p(x),p(y)\}.$
 \vspace{2mm}

 A 1-monotone seminorm is called \emph{monotone}. Also a seminorm   which is $\gamma$-monotone for some $\gamma>0$ is called sometimes semi-monotone (see \cite{Deimling}).

 These properties can be characterized in terms of the Minkowski functional attached to an absorbing subset $A$ of a vector space $X$, given by
 \begin{equation}\label{def.Mink-fc}
 p_A(x)=\inf\{t>0 : x\in tA\},\quad (x\in X.)
 \end{equation}

 It is well known that if the set $A$ is absolutely convex and absorbing, then $p_A$ is a seminorm on $X$ and
 \begin{equation}
\ainter(A)= \{x\in X : p_A(x)<1\}\subset A\subset \{x\in X : p_A(x)\le 1\}=\aclos(A).\end{equation}

 \begin{prop}[\cite{BreckW},\, Proposition 2.5.6]\label{p1.Mink-fc-monot}
 Let $A$ be an absorbing absolutely convex subset of an ordered vector space $X$.
 \begin{enumerate}
\item[{\rm 1.}]  If $\,[A]\subset \gamma A,\,$ then the seminorm $p_A$ is $\gamma$-normal.
\item[{\rm 2.}]  If $\,\forall y\in A,\; [0;y]\subset \gamma A,\,$ then the seminorm $p_A$ is $\gamma$-monotone.
\item[{\rm 3.}]   If $\,\forall y\in A,\; [-y;y]\subset \gamma A,\,$ then the seminorm $p_A$ is $\gamma$-absolutely monotone.
 \end{enumerate}\end{prop}

 Based on Theorem \ref{t1.char-normal-cone} and Proposition \ref{p1.Mink-fc-monot} one can give further characterizations of normal cones in LCS.

 \begin{theo}[\cite{BreckW},\, \cite{Peressini} and \cite{Schaef}]\label{t2.char-normal-cone}
 Let $(X,\tau)$ be a  LCS ordered by a cone $K.$ The following are equivalent.
   \begin{enumerate}
\item[{\rm 1.}] The cone $K$ is normal.
\item[{\rm 2.}] There exists $\gamma>0$ and a family   of $\gamma$-normal seminorms generating the topology $\tau$ of $X$.
\item[{\rm 3.}] There exists $\gamma>0$ and a family   of $\gamma$-monotone seminorms generating the topology $\tau$ of $X$.
\item[{\rm 4.}] There exists $\gamma>0$ and a family   of $\gamma$-absolutely monotone  seminorms generating the topology $\tau$ of $X$.
    \end{enumerate}

 All the above equivalences hold also with $\gamma =1$ in all places.
 \end{theo}

\subsection{Normal cones in normed spaces}

 We shall consider now characterizations of normality in the case of normed spaces.
 For a normed space $(X,\|\cdot\|),$ let  $B_X=\{x\in X : \|x\|\le 1\}$ be its closed unit ball and   $S_X=\{x\in X : \|x\|= 1\}$   its unit sphere.

 \begin{theo}[\cite{Deimling} and \cite{Guo}]\label{t4.char-normal-cone}
   Let $K$ be a cone in a normed space $(X,\|\cdot\|).$  The following are equivalent.
   \begin{enumerate}
     \item[{\rm 1.}] The cone $K$ is normal.
      \item[{\rm 2.}] There exists a monotone norm $\|\cdot\|_1$ on $X$ equivalent to the original norm $\|\cdot\|.$
 \item[{\rm 3.}]   For all sequences $(x_n),\, (y_n),\, (z_n)$ in $X$ such that $x_n\le z_n\le y_n,\,n\in\mathbb{N},$ the conditions $\lim_nx_n=x=\lim_ny_n$ imply $\lim_nz_n=x.$
 \item[{\rm 4.}] The  order--convex hull $[B_X]$   of the unit ball is bounded.
 \item[{\rm 5.}]   The order interval $[x;y]_o$ is bounded for every $x,y\in X.$
  \item[{\rm 6.}] There exists $\delta > 0$ such that $\forall x,y\in K\cap S_X,\; \|x+y\|\ge\delta.$
   \item[{\rm 7.}] There exists $\gamma > 0$ such that $\,\forall x,y\in K,\; \|x+y\|\ge\gamma\max\{\|x\|,\|y\|\}.$ \item[{\rm 8.}]  There exists $\lambda > 0$ such that $\, \|x\|\le\lambda \|y\|,$ for all $\,x,y\in K\,$ with $x \le y.$
       \end{enumerate} \end{theo}

       We notice also the following result, which can be obtained as  a consequence of a result of T. And\^o on ordered locally convex spaces (see \cite[Theorem 2.10]{Alipr}).
\begin{prop}[\cite{Alipr}, Corollary 2.12]\label{p.ord-B-sp}
Let $X$ be a Banach space ordered by a   generating cone $X_+$ and $B_X$ its closed unit ball. Then
$
 \,(B_X \cap X_+)- (B_X \cap X_+)\,$
is a neighborhood of 0.
  \end{prop}
  \subsection{Completeness and order completeness in ordered TVS}
The following notions are inspired by Cantor's theorem on the convergence of bounded monotone sequences of real numbers.

Let $X$ be a Banach space ordered by a cone $K$. The cone $K$ is called:
\vspace{2mm}

\textbullet \quad \emph{regular} if every increasing and order--bounded sequence in $X$ is convergent;

\textbullet \quad \emph{fully regular} if every increasing and norm-bounded sequence in $X$ is convergent.

\vspace{2mm}

By Proposition \ref{p.Dedekind} if $X$ is a regular normed lattice, then every countable subset of $X$ has a supremum.

These notions are related in the following way.
\begin{theo}[\cite{Guo}, Theorems 2.2.1 and 2.2.3]\label{t.regular-cone}
  If $X$ is a Banach space ordered by a cone $K$, then
  $$
  K\; \mbox{fully regular} \; \Longrightarrow\; K\; \mbox{regular} \; \Longrightarrow\; K\; \mbox{normal}.$$

  If the Banach  space $X$ is reflexive, then the reverse implications hold too, i.e. both implications become equivalences.
\end{theo}

Some relations between completeness and order completeness in ordered topological vector spaces were obtained by Ng \cite{ng72}, Wong \cite{wong72} (see also the book \cite{Wong-Ng}). Some questions about  completeness in ordered metric spaces are discussed by Turinici \cite{turinici80}.

Let $(X,\tau)$ be a TVS  ordered by a cone $K$. One says that the space $X$ is
\vspace{2mm}

\textbullet \quad \emph{fundamentally $\sigma$-order complete} if every increasing $\tau$-Cauchy sequence in $X$ has a supremum;

\textbullet \quad \emph{monotonically sequentially complete} if every increasing $\tau$-Cauchy sequence in $X$ is convergent in $(X,\tau)$.
\vspace{2mm}

In the following propositions $(X,\tau)$ is a TVS ordered by a cone $K.$

The following result is obvious.

\begin{prop}\label{p1.o-complete} \hfill\begin{enumerate}
 \item[{\rm 1.}] If $ X$ is   sequentially complete, then $X$ is monotonically sequentially complete.
 \item[{\rm 2.}] If $ X$ is monotonically sequentially complete, then $X$ is fundamentally $\sigma$-order complete.
 \item[{\rm 3.}] If $ K $ is normal and generating, and $X$ is fundamentally $\sigma$-order complete, then $X$ is monotonically sequentially complete.
\end{enumerate}\end{prop}

The following characterizations of these completeness conditions will be used in the study of the completeness with respect to the Thompson metric.
\begin{prop}\label{p2.o-complete}
  The following conditions are equivalent.
  \begin{enumerate}
\item[{\rm 1.}] $X$ is fundamentally $\sigma$-order complete.
\item[{\rm 2.}] Any decreasing Cauchy sequence in $X$ has an infimum.
\item[{\rm 3.}] Any increasing Cauchy sequence in $K$ has a supremum.
\item[{\rm 4.}]  Any decreasing Cauchy sequence in $K$ has an infimum.
\end{enumerate}
\end{prop}
\begin{prop}\label{p3.o-complete}
  The following conditions are equivalent.
  \begin{enumerate}
\item[{\rm 1.}] $X$ is monotonically sequentially complete.
\item[{\rm 2.}] Any decreasing Cauchy sequence in $X$ has limit.
\item[{\rm 3.}] Any increasing Cauchy sequence in $K$ has limit.
\item[{\rm 4.}]  Any decreasing Cauchy sequence in $K$ has limit.
\end{enumerate}
\end{prop}

\begin{prop}\label{p4.o-complete}
If $K$ is lineally solid, then the following conditions are equivalent.
  \begin{enumerate}
\item[{\rm 1.}] $X $ is fundamentally $\sigma$-order complete.
\item[{\rm 2.}] Any increasing Cauchy sequence in $\ainter(K)$ has a supremum.
\item[{\rm 3.}] Any decreasing Cauchy sequence in $\ainter(K)$ has an infimum.
  \end{enumerate}
  \end{prop}

  \begin{prop}\label{p5.o-complete}
If K is lineally solid, then the following conditions are equivalent.
  \begin{enumerate}
\item[{\rm 1.}] $X$ is monotonically sequentially complete.
\item[{\rm 2.}] Any increasing $\tau$-Cauchy sequence in $\ainter(K)$ has limit.
\item[{\rm 3.}] Any decreasing $\tau$-Cauchy sequence in $\ainter(K)$ has limit.
  \end{enumerate}
  \end{prop}

\section{The Thompson metric}\label{S.Thompson-metric}
\subsection{Definition and fundamental properties}\label{Ss.def-T-metric}

Let $X$ be a  vector  space and $K$ a  cone in $X$. The relation
\begin{equation}\label{def.linked}
x\sim y \iff \exists \lambda,\mu >0,\;\; x\le \lambda y\;\mbox{and}\; y \le \mu x,
\end{equation}
is an equivalence relation in $K$. One says that two elements $x,y\in K$ satisfying  \eqref{def.linked} are  \emph{linked} and the equivalence classes are called \emph{components}. The equivalence class of an element $x\in K $ will be denoted by $K(x).$
\begin{prop}\label{p1.equiv-cone}
Let $X$ be a vector space ordered by a cone $K$.
\begin{enumerate}
\item[{\rm 1.}] \; $K(0)=\{0\}$ and $\ainter(K)$ is a component of $K$ if $K$ is lineally solid.
\item[{\rm 2.}] Every  component  $Q$ of $K$ is  order--convex, convex, closed under addition and multiplication by positive scalars, that is  $Q\cup\{0\}$ is an order--convex cone.
\end{enumerate}
\end{prop}\begin{proof}
  We justify only the assertion concerning $\ainter K$, the others being trivial. If $x,y\in \ainter K, $ then there exist $\alpha,\beta>0$ such that $x+t y\in K$ for all $t\in[-\alpha,\alpha]$ and $y+sx\in K$ for all $s\in[-\beta,\beta].$ It follows $y-\beta x\in K,$ i.e. $y\ge \beta x,$ and $x-\alpha  y\in K,$ i.e. $x\ge \alpha y.$
\end{proof}

For two linked elements $x,y\in K$ put
  \begin{equation}\label{def1.t-metric}
 \sigma(x,y)=\{s\ge 0 : e^{-s}x\le y\le e^sx\},
 \end{equation}
 and let
 \begin{equation}\label{def2.t-metric}
 d_T(x,y)=\inf\sigma(x,y) .\end{equation}

 \begin{remark}\label{re.extended-T-metric}
It is convenient to define $d_T$ for any pair of elements in $K$, by setting  $d_T(x,y)=\infty$
for any $x,y$ not lying in the same component  of $K$  which, by   \eqref{def2.t-metric},  is  in concordance with the usual convention   $\inf\emptyset =\infty.$
 In this way,  $d_T$  becomes an extended (or generalized) (semi)metric (in the sense of Jung \cite{jung69})
on $K$ and,  for all $x,y\in K,\, x\sim y \iff d(x,y) < \infty.$ Though  $d_T$  is not a usual (semi)metric on the whole cone, we will continue to call  $d_T$ a metric. The   Thompson  metric is also called, by some authors, the part
metric (of the cone $K$).
\end{remark}

\begin{remark} \label{re2.T-metric}It is obvious that the definition of $d(x,y)$ depends only on the ordering of the vector subspace spanned by $\{x,y\}$. This ensures that if $x$ and $y$ are seen as elements of some vector subspace $Y$ of $X$, then $d_T(x,y)$ is the same in $Y$ as in $X$ (assuming, of course, that $Y$ inherits the ordering from $ X$).
\end{remark}

 The initial approach of Thompson \cite{thomp63} was slightly different. He considered the set
\begin{equation}\label{def3.t-metric}
  \alpha(x,y)= \{\lambda \ge 1 : x\le \lambda y\} .
   \end{equation}
and defined the distance between $x$ and $y$ by
 \begin{equation}\label{def4.t-metric}
 \delta(x,y)=\ln\left(\max\{\inf\alpha(x,y),\inf\alpha(y,x)\}\right).\end{equation}

 The following proposition shows that the  relations \eqref{def2.t-metric} and \eqref{def4.t-metric} yield the same function.
 \begin{prop}\label{p1.T-metric}
 For every $x,y\in K$ the following equality holds
 $$
 d_T(x,y)=\delta(x,y).$$
  \end{prop}
  \begin{proof} It suffices to prove the equality for two linked elements $x,y\in K.$ In this case let
  $$\alpha_1=\inf\alpha(x,y),\;\; \alpha_2=\inf\alpha(y,x)\;\; \mbox{and}\;\; \alpha=\max\{\alpha_1,\alpha_2\}.$$

  Put also
 $$
    d=d_T(x,y)=\inf\sigma(x,y)\quad\mbox{and}\quad \delta=\delta(x,y)=\ln\alpha.
$$

For    $s\in\mathbb{R}$ let  $\lambda =e^s.$ Then the following  equivalences hold
\begin{equation}\label{eq2.p1.T-metric}\begin{aligned}
 s\in\sigma(x,y)\iff& \lambda^{-1}x\le y\le\lambda x\\ \iff& x\le\lambda y\;\wedge \; y\le \lambda x \iff \lambda\in\alpha(x,y)\cap\alpha(y,x) .
\end{aligned}\end{equation}

 Consequently $\lambda\ge\max\{\alpha_1,\alpha_2\}=\alpha$ and $s\ge \ln\alpha=\delta,$ for every $s\in \sigma(x,y),$ and so

 \begin{equation}\label{eq1.p1.T-metric}
 d=\inf\sigma(x,y)\ge \delta.
 \end{equation}

 To prove the reverse inequality, suppose that $\alpha_1\ge \alpha_2$ and let $\lambda>\alpha_1$.
 Then $\lambda\in\alpha(x,y)\cap\alpha(y,x)$ and the   equivalences \eqref{eq2.p1.T-metric} show that $s=\ln\lambda\in \sigma(x,y),$ so that
 $\ln\lambda \ge d.$ It follows
 $$
 \delta=\inf\{\ln\lambda : \lambda >\alpha_1\}\ge d ,$$
 which together with \eqref{eq1.p1.T-metric}  yields $\delta=d.$
  \end{proof}

    There is another metric defined on the components of $K$, namely the \emph{Hilbert projective metric},  defined by
 \begin{equation}\label{def.Hilb-metric}
 d_H(x,y)=\ln\left(\inf\alpha(x,y)\cdot\inf\alpha(y,x)\}\right),
 \end{equation}
 for any two linked elements $x,y$ of $K$.

 The term projective comes from the fact that $d_H(x,y)=0$ iff $x=\lambda y$ for some $\lambda>0.$

   The original Hilbert's definition (see \cite{hilbert1895})  of the  metric was the following. Consider an open bounded convex subset $\Omega$ of the Euclidean space $\mathbb{R}^n$.    For two points
$x,y \in \Omega$ let $\ell_{xy}$ denote the straight line through $x$ and $y$, and denote the
points of intersection of $\ell_{xy}$ with the boundary $\partial\Omega$ of $\Omega$ by $x',y',$ where $x$ is between $x'$ and $y$, and $y$ is between $x$ and $y'$. For $x\ne y$ in $\Omega$ the Hilbert distance between $x$ and $y$ is defined by
\begin{equation}
\label{def.Hilb-metric-R2}
\delta_H(x,y)=\ln\left(\frac{\|x'-y\|\cdot\|y'-x\|}{\|x'-x\|\cdot\|y'-y\|}\right)\,,
\end{equation}
and
$\delta_H(x, x) = 0$  for all $x\in\Omega,$    where $\|\cdot\|$ stands for the Euclidean norm in $\mathbb{R}^n.$ The metric space $(\Omega,\delta_H)$ is
called the Hilbert geometry on $\Omega$. In this geometry there exists a triangle with non-colinear vertices such that the sum of the lengths of two sides   equals    the length of the third side. If $\Omega$ is the open unit disk, the Hilbert metric
is exactly the Klein model of the hyperbolic plane.

The definition \eqref{def.Hilb-metric} of Hilbert metric on cones in vector spaces was proposed by Bushell \cite{bushell73a} (see also \cite{bushell73b}).

\begin{note}
 As we shall consider only the Thompson metric, the subscript $T$ will be omitted, that is $d(\cdot,\cdot)$ will stand always for the Thompson metric.
\end{note}

  In the following proposition we collect some properties of the set $\sigma(x,y).$

  \begin{prop}\label{p2.T-metric}
  Let $X$ be a  vector space ordered by a cone $K$ and $x,y,z$ linked  elements in $K$.
 \begin{enumerate}
\item[{\rm 1.}] Symmetry:\quad $\sigma(y,x)=\sigma(x,y).$
\item[{\rm 2.}]  $(d(x,y);\infty)\subset \sigma(x,y)\subset [d(x,y);\infty).\;$ If the cone $K$ is Archimedean, then  $d(x,y)\in\sigma(x,y),$ that is $\sigma(x,y)=[d(x,y);\infty)$.
 \item[{\rm 3.}]  $\sigma(x,y)+\sigma(y,z)\subset\sigma(x,z).$
 \end{enumerate} \end{prop}
 \begin{proof} 1.\; The symmetry follows from the definition of the set $\sigma(x,y).$

 2.\; The inclusion $(d(x,y);\infty)\subset \sigma(x,y)$ follows from the fact that $0<\lambda<\mu$ and $x\ge 0 $ implies $ \lambda x\le \mu x.$ The second inclusion follows from the fact that   no $\lambda<d(x,y)$ belongs to $\sigma(x,y).$

Let $d=d(x,y)=\inf\sigma(x,y). $ Since an   Archimedean cone   is lineally closed  and   $y-e^{-s}x\in K$ for every $s>d$, it follows  $y-e^{-d}x\in K.$ Similarly $e^{d}x-y\in K,$ showing that $d\in \sigma(x,y).$

3.\; Let $s\in\sigma(x,y)$ and  $t\in\sigma(y,z).$ Then
$$
e^{-s}x\le y\le e^sx\quad\mbox{and}\quad e^{-t}y\le z\le e^ty.$$

It follows
$$
e^{-(s+t)}x\le e^{-t}y\le z\quad\mbox{and} \quad z\le e^{t}y \le e^{s+t}x,$$
which shows that $s+t\in \sigma(x,z).$
 \end{proof}

 Now it is easy to show that   the  function $d$ given by \eqref{def2.t-metric} is an extended  semimetric.

 \begin{prop} \label{p3.T-metric} Let $X$ be a  vector space ordered by a cone $K$.
 \begin{enumerate}
   \item[{\rm 1.}] The function $d$ defined by \eqref{def2.t-metric} is a semimetric on each component of $K$.
   \item[{\rm 2.}] The function $d$   is a metric on each component of $K$ iff the order defined by the cone $K$ is almost Archimedean.
 \end{enumerate} \end{prop}
 \begin{proof} 1.\; The fact that $d$ is a semimetric follows from the properties of the sets $\sigma(x,y)$ mentioned in  Proposition \ref{p2.T-metric}.

 2. \; Suppose now that the cone $K$ is almost Archimedean and $d(x,y)=0$ for two linked elements $x,y\in K.$ It follows
 $$
 \begin{aligned}
   \forall s>0,\;\; e^{-s}x\le y\le e^sx \iff&  \forall s>0,\;\; (e^{-s}-1)x\le y-x\le (e^s-1)x\\
   \iff& \forall s>0,\;\; -\frac{e^{s}-1}{e^s}x\le y-x\le (e^s-1)x .
 \end{aligned}$$

  The inequality $ e^{-s}(e^{s}-1) \le e^s-1$ implies $-e^{-s}(e^{s}-1)x\ge -(e^s-1)x$. Consequently,
  $$
  \forall s>0,\;\; -(e^{s}-1)x\le y-x\le (e^s-1)x \iff \forall \lambda >0,\;\; -\lambda x\le y-x\le \lambda x.$$

 Taking into account that $K$ is almost Archimedean it follows $y-x=0,$ that is $y=x.$

 To prove the converse, suppose that $K$ is not almost Archimedean. Then there exists a line $D=\{x+\mu y : \mu\in\mathbb{R}\},\, $ with $\, y\ne 0,$ contained in $K$.  If $x=0, $ then $\pm y\in K,$ that would imply $y=0, $ a contradiction.

 Consequently $x\ne 0.$  Observe   that in this case, for all $\mu\in\mathbb{R},$
 \begin{equation}\label{eq1.p3.T-metric}
 d(x,x+\mu y)=0 ,
 \end{equation}
which  shows that $d$ is not a metric.
The equality \eqref{eq1.p3.T-metric} is equivalent to
 \begin{equation}\label{eq2.p3.T-metric}
 \forall s>0,\;\; e^{-s}x\le x+\mu y\le e^sx.\end{equation}

 The inclusion $D\subset K$ implies $x\pm \lambda y\in K$ for all $\lambda>0,$ and so
 $$
 -x\le \lambda y\le x,$$
 for all $\lambda >0.$ Taking $\lambda =\mu (1-e^{-s})^{-1}$ the first inequality from above becomes
 $$
-(1- e^{-s})x\le \mu y \iff e^{-s}x\le x+\mu y .$$

 From the second inequality one obtains
$$
\mu y\le (1-e^{-s})x=e^{-s}(e^s-1)x\le (e^s-1)x,$$
which implies
$$
x+\mu y\le e^s x,$$
showing that the inequalities \eqref{eq2.p3.T-metric} hold.
\end{proof}

\begin{remark}\label{re1.T-metric}
By the triangle inequality, the equality \eqref{eq1.p3.T-metric} implies that $d(u,v)=0$ for any two points $u,v$ on $D$, that is
$$
 d(x+\lambda y,x+\mu y)=0,$$
 for all $\lambda, \mu\in\mathbb{R}.$
\end{remark}

\begin{example}
 If $X=\mathbb{R}^n$ and $K=\mathbb{R}^n_+$, then the components of $K$ are $\{0\},\, (0;\infty)\cdot e_i,\, 1\le i\le n,$ and $\ainter (K) =\{x\in K : x_i>0,\, i=1,\dots,n\},$  while $d(x,y)=\max\{|\ln x_i-\ln y_i| : 1\le i\le n\},$  for any $x=(x_i)_{i=1}^n $ and $y=(y_i)_{i=1}^n $, with $x_i,y_i > 0,\, i=1,\dots,n$.
   \end{example}

The following proposition contains some further properties of the sets $\sigma(x,y)$ and their corespondents for the Thompson metric.

\begin{prop} \label{p4.T-metric}
Let $X$ be a  vector space ordered by a cone $K$.
 \begin{enumerate}
   \item[{\rm 1.}]  For $x,y\in K$ and $\lambda,\mu >0$
   \begin{align*}
 {\rm (i)}&\quad \sigma(\lambda x,\lambda y)= \sigma(x,y)\quad\mbox{and so}\quad d(\lambda x,\lambda y)= d(x,y);\\
 {\rm (ii)}&\quad \sigma(\lambda x,\mu x)= [\big|\ln\big(\frac\lambda\mu\big)\big|;\infty)\quad\mbox{and so}\quad d(\lambda x,\mu x)= \big|\ln\big(\frac\lambda\mu\big)\big|;\\
 {\rm (iii)}& \quad \mbox{If}\;\; \mu x\le y\le \lambda x, \quad\mbox{for some}\quad \lambda,\mu>0,\quad\mbox{then}\quad
d(x,y)\le\ln\max\{\mu^{-1},\lambda\}.
   \end{align*}
   \item[{\rm 2.}] If $\sigma(x,y)\subset\sigma(x',y'),$ then $d(x,y)\ge d(x',y').$ The converse is true if the order is Archimedean. Also
      \begin{equation}\label{eq1.p4.T-metric}
\max\{d(x,y),d(x',y')\}=\inf[\sigma(x,y)\cap \sigma(x',y')].
\end{equation}
\item[{\rm 3.}]   The following monotony inequalities  hold
\begin{equation}\label{eq2.p4.T-metric}\begin{aligned}
{\rm(i)}\quad x\le x'\;\mbox{and}\; y'\le y\;&\Longrightarrow\; d(x',x'+y')\le d(x,x+y);\\
{\rm(ii)}\quad\;\; x\le x'\le y'\le y\;&\Longrightarrow\; d(x',y')\le d(x,y).
\end{aligned} \end{equation}
 \item[{\rm 4.}]   For all $x,y,x',y'\in K$ and $\lambda,\mu >0,$
\begin{equation}\label{eq3.p4.T-metric}
   d(\lambda x+\mu y,\lambda x'+\mu y')\le \max\{d(x,x'),d(y,y')\}.
   \end{equation}
 \end{enumerate} \end{prop}
 \begin{proof}
 1. The equalities  from  (i) are obvious.

 To prove (ii) suppose $\lambda >\mu.$ Then  $e^{-s}x\le x\le \lambda\mu^{-1}x$ implies $ \mu e^{-s} x\le \lambda x,$
 for every $s>0.$ Since
 $$
 \lambda\mu^{-1}x \le e^sx \iff s\ge \ln\big(\lambda\mu^{-1}\big),$$
 it follows $\sigma(\lambda x,\mu x)=\big[\ln\big(\lambda\mu^{-1}\big);\infty)$ and $d(\lambda x,\mu x)= \ln\big(\lambda\mu^{-1}\big).$

To prove (iii) observe that  $\mu x\le y$ is equivalent to $x\le \mu^{-1} y,$ that is $\mu^{-1}\in\alpha(x,y),$ and so $\mu^{-1}\ge \inf\alpha(x,y).$ Similarly, $y\le \lambda x$ is equivalent to $\lambda\in\alpha(y,x),$ implying $\lambda\ge\inf\alpha(x,y).$  It follows
 $$
 \ln\max\{\mu^{-1},\lambda\}\ge \ln(\max\{\inf\alpha(x,y),\inf\alpha(y,x)\}=d(x,y).$$

 2.\; The first implication is obvious. The converse follows from the fact that $\sigma(x,y)=[d(x,y);\infty)$ and
 $\sigma(x',y')=[d(x',y');\infty)$ if $K$ is Archimedean (Proposition \ref{p2.T-metric}.2).

  The equality \eqref{eq1.p4.T-metric} follows from the inclusions
\begin{align*}
  &(d(x,y);\infty)\subset \sigma(x,y)\subset [d(x,y);\infty)\quad\mbox{and}\\
  &(d(x',y');\infty)\subset \sigma(x',y')\subset [d(x',y');\infty).
\end{align*}

3.\;  The inequality (i) for the metric $d$ will follow from the inclusion
\begin{equation}\label{eq4.p4.T-metric}
\sigma(x,x+y)\subset \sigma(x',x'+y').\end{equation}

Let $s\in \sigma(x,x+y),$ that is $s>0$ and
\begin{equation*}
e^{-s}x\le x+y\le e^s x.\end{equation*}

Then
\begin{align*}
  e^{-s}x'\le& x'\le x'+y'\le x'+y=x+y +(x'-x)\\
  \le&e^{s}x+ e^{s}(x'-x)=e^{s}x',
\end{align*}
showing that $s\in \sigma(x',x'+y').$ 

The inequality (ii) follows from (i) by taking $y:=y-x\ge y'-x'=:y'.$

4.\;   By 1.(i), $\,d(\lambda x,\lambda x')=d(x,x')$ and $d(\mu y ,\mu y')=d(y,y'),$  so that it is  sufficient to show that
\begin{equation}\label{eq3b.p4.T-metric}
d(x+y,x'+y')\le \max\{d(x,x'),d(y,y')\}. \end{equation}

Taking into account  \eqref{eq1.p4.T-metric} and the assertion 2 of the proposition, the inequality \eqref{eq3b.p4.T-metric} will be a consequence of the inclusion
$$
\sigma(x,x')\cap \sigma(y,y')\subset \sigma(x+y,x'+y') . $$

But, if  $s\in \sigma(x,x')\cap \sigma(y,y'),$ then $e^{-s}x\le x'\le e^{s}x $ and $e^{-s}y\le y'\le e^{s}y, $ which by addition yield $e^{-s}(x+y)\le x'+y'\le e^{s}(x + y),$ that is $s\in \sigma(x+y,x'+y').$
\end{proof}

Based on these properties one obtains other properties of the Thompson metric.
\begin{theo}\label{t1.T-metric}
Let $X$ be a  vector space ordered by a cone $K$.
\begin{enumerate}
\item[{\rm 1.}] The function $d$ is quasi-convex with respect to each of its argument, that is
\begin{equation}\label{eq1.q-cv}\begin{aligned}
&d((1-t) x+t y,v)\le\max\{d(x,v),d(y,v)\}\quad\mbox{and}  \\
&d(u,(1-t) x+t y)\le\max\{d(u,x),d(u,y)\},
\end{aligned}\end{equation}
for all $x,y,u,v\in K$ and $t\in[0;1].$
\item[{\rm 2.}] The following convexity--type inequalities hold
 \begin{equation}\label{eq2.q-cv}\begin{aligned}
&d((1-t)x+t y,v)\le\ln\left((1-t)e^{d(x,v)}+te^{d(y,v)}\right) ,   \\
&d(u,(1-t)x+t y)\le\ln\left((1-t)e^{d(u,x)}+te^{d(u,y)}\right),
\end{aligned}\end{equation}
for all $x,y,u,v\in K$ and $t\in[0;1],$ and
  \begin{equation}\label{eq3.q-cv}\begin{aligned}
&d((1-t)x+t y,(1-s)x+s y)\le\ln\left(|s-t)e^{d(x,y)}+1-|s-t|\right),
\end{aligned}\end{equation}
for all $x,y \in K,\, x\sim y,\,$ and $s,t\in [0;1].$
\end{enumerate}
\end{theo}\begin{proof}
  1.\; By \eqref{eq1.p4.T-metric} and Proposition \ref{p4.T-metric}.1.(i),
\begin{align*}
 d((1-t)x+t y,v) =& d((1-t)x+t y,(1-t)v+t v) \\\le& \max\{d((1-t)x,(1-t)v),d(t y,tv)\}= \max\{d(x,v),d(y,v)\},
\end{align*}
showing that  the first inequality in \eqref{eq1.q-cv} holds. The second one follows by the symmetry of the metric $d.$

2.\; For  $s_1\in \sigma(x,v) $ and $s_2\in \sigma(y,v) $ put $s=\ln\left((1-t)e^{s_1}+ te^{s_2}\right). $   By a straightforward calculation it follows that
$$
\left((1-t)e^{s_1}+ te^{s_2}\right)\cdot\left((1-t)e^{-s_1}+ te^{-s_2}\right)=2t(1-t)(\cosh(s_1-s_2)-1)\ge 0,$$
which implies
$$
-s\le  \ln\left((1-t)e^{-s_1}+ te^{-s_2}\right),$$
or, equivalently,
\begin{equation*}
e^{-s}\le (1-t)e^{-s_1}+ te^{-s_2}.\end{equation*}

The above inequality and the inequalities $e^{-s_1}v\le x,\,  e^{-s_2}v\le y$ imply
$$
e^{-s}v\le \left((1-t)e^{-s_1}+ te^{-s_2}\right) v\le (1-t)x+ty.$$

Similarly, the inequalities $x\le e^{s_1}v,\, y\le e^{s_2}v, $ and the definition of $s$ imply
$$
  (1-t)x+ty\le   \left((1-t)e^{-s_1}+ te^{-s_2}\right)v = e^s v.$$

  It follows $s\in \sigma((1-t)x+ty,v)$  and so

$$
  d((1-t)x+ty,v)\le  s = \ln\left((1-t)e^{-s_1}+ te^{-s_2}\right) ,$$
for all $s_1\in\sigma(x,v)$ and all $s_2\in \sigma(y,v).$ Passing to infimum with respect to $s_1$ and $s_2, $  one obtains the first inequality  in \eqref{eq2.q-cv}. The second inequality follows by the symmetry of $d$.

It is obvious that \eqref{eq3.q-cv} holds for $s=t$, so we have to prove it only for $s\ne t.$ By symmetry it suffices to consider only the case  $t>s.$ Putting  $z_t=(1-t)x+t y$ and  $z_s=(1-s)x+s y,$ it follows
$z_s=(1-\frac st)x+\frac st z_t,$ so that, applying twice the inequality \eqref{eq2.q-cv},
$$
d(x,z_t)\le \ln \left(1-t+te^{d(x,y)}\right),$$
and
\begin{align*}
d(z_s,z_t)\le& \ln\left((1-\frac st)e^{d(x,z_t)}+\frac st\right)\\
\le& \ln\left((1-\frac st)(1-t+te^{d(x,y)}) +\frac st\right)=\ln\left((t - s) e^{d(x,y)}+1-(t-s)\right).
\end{align*}
\end{proof}

Recall that a metric space $(X,\rho)$ is called \emph{metrically convex} if for every pair of distinct points $x,y\in X$ there exists a point $z\in X\setminus\{x,y\}$ such that
\begin{equation}\label{def.metric-cv}
\rho(x,y)=\rho(x,z)+\rho(z,y).\end{equation}

The following theorem,  asserting  that every component of $K$ is metrically convex with respect to the Thompson metric, is a slight extension of a result of Nussbaum \cite[Proposition 1.12]{Nuss88}.
\begin{theo}\label{t1.metric-cv}
  Every component of $K$ is metrically convex with respect to the Thompson metric $d$. More exactly, for every pair of distinct points $x,y\in X$ and every $t\in(0;1)$ the point
  $$
  z=\frac{\sinh r(1-t)}{\sinh r}x+ \frac{\sinh rt}{\sinh r} y,$$
 where $r=d(x,y)$, satisfies \eqref{def.metric-cv}.
\end{theo} \begin{proof}
By the triangle inequality it suffices to show that
\begin{equation}\label{eq2.metric-cv}
r=d(x,y)\ge d(x,z)+d(z,y).
\end{equation}

If  $s\in\sigma(x,y),$ that is $e^{-s}x\le y\le e^sx,$
then
\begin{equation}\label{eq3.metric-cv}\begin{aligned}
&\left(\frac{\sinh r(1-t)}{\sinh r}+ \frac{\sinh rt}{\sinh r}e^{-s}\right)x\le z\le  \left(\frac{\sinh r(1-t)}{\sinh r} + \frac{\sinh rt}{\sinh r} e^s\right)x .
\end{aligned}\end{equation}

Putting
$$
\mu(s)=\frac{\sinh r(1-t)}{\sinh r}+ \frac{\sinh rt}{\sinh r}e^{-s} \quad\mbox{and}\quad
\lambda(s)=\frac{\sinh r(1-t)}{\sinh r}+ \frac{\sinh rt}{\sinh r}e^{s},$$
the inequalities \eqref{eq3.metric-cv} imply
$$
d(x,z)\le\ln(\max\{\mu(s)^{-1},\lambda(s)\}),$$
for all $s>r.$ Since the functions $\mu(s)^{-1}$ and $\lambda(s)$ are both continuous on $(0;\infty)$, it follows
\begin{equation}\label{eq4.metric-cv}
d(x,z)\le\ln(\max\{\mu(r)^{-1},\lambda(r)\}).\end{equation}

Taking into account the definition of the function $\sinh,$ a direct calculation shows that $\mu(r)^{-1}=\lambda(r)=e^{rt}$, and so  the inequality \eqref{eq4.metric-cv} becomes
\begin{equation*}
d(x,z)\le rt.
\end{equation*}

By symmetry
\begin{equation*}
d(z,y)\le r(1-t),
\end{equation*}
so that \eqref{eq2.metric-cv} holds.
\end{proof}

\subsection{Order-unit seminorms}\label{Ss.u-semin}

Suppose that $X$ is a vector space ordered by a cone $K$. For $u\in K\setminus\{0\}$ put
\begin{equation}\label{def.Xu}
X_u=\cup_{\lambda\ge 0}\lambda [-u;u]_o.\end{equation}

It is obvious that $X_u$ is a nontrivial subspace of $X\; (\mathbb{R} u\subset X_u),$ and that $[-u;u]_o$ is an absorbing absolutely convex subset of $X_u$ and so $u$ is a unit in the ordered vector space $(X_u,K_u)$, where    $K_u$ is the cone in $X_u$ given by
\begin{equation}\label{def.Ku}
K_u=K\cap X_u\,,
\end{equation}
or, equivalently, by
\begin{equation}\label{char.Ku}
K_u=\cup_{\lambda \ge 0}\lambda[0;u]_o\,.
\end{equation}

The Minkowski functional
\begin{equation}\label{def.normu}
|x|_u=\inf\{\lambda > 0 : x\in\lambda[-u;u]_o\},
\end{equation}
  corresponding to the set $[-u;u]_o,$   is a seminorm on the space $X_u$ and
\begin{equation}\label{eq.u-sn}
|-u|_u=|u|_u =1.
\end{equation}

For convenience, denote by the subscript $u$ the topological notions corresponding to the seminorm $|\cdot|_u$. Let also $B_u(x,r),\, B_u[x,r]$ be the open, respectively closed, ball with respect to $|\cdot|_u$. For $x\in X_u$ let
\begin{equation}\label{def.Mink}
\mathcal M_u(x)=\{\lambda > 0 : x\in\lambda[-u;u]_o\},\end{equation}
so that
$$
|x|_u=\inf\mathcal M_u(x).$$

Taking into account   the convexity of $[-u;u]_o$ it follows that
\begin{equation}\label{eq-Mink-u-sn}
(|x|_u;\infty)\subset \mathcal M_u(x)\subset [|x|_u;\infty),
\end{equation}
for every $x\in X_u.$

\begin{prop}\label{p1.Ku} Let $u\in K\setminus\{0\}$ and $X_u, K_u, |\cdot|_u$ as above.
\begin{enumerate}
\item[{\rm 1.}] If $v\in K$ is linked to $u$, then $  X_u = X_v,\, K_u=K_v$ and the seminorms $|\cdot|_u, |\cdot|_v$ are equivalent. More exactly the following inequalities hold for all $x\in X_u$
    \begin{equation}\label{eq.p1.Ku-equiv}
    |x|_u\le |v|_u |x|_v\quad\mbox{and}\quad |x|_v\le |u|_v |x|_u.
    \end{equation}
\item[{\rm 2.}] The Minkowski functional $|\cdot|_u$ is a norm on $X_u$ iff the cone $K_u$ is almost Archimedean.
\item[{\rm 3.}] The seminorm $|\cdot|_u$ is monotone: $x,y\in X_u$ and $0\le x\le y$ implies $|x|_u\le |y|_u.$
\item[{\rm 4.}] The cone $K_u$ is generating and normal in $X_u\, .$
\item[{\rm 5.}] For any $x\in X_u$ and $r>0,\; B_u(x,r)\subset x+r[-u;u]_o\subset B_u[x,r].$
\item[{\rm 6.}] The following equalities hold:
\begin{equation}\label{eq1.p1.Ku}
\ainter(K_u) = K(u) = \inter_u(K_u) .\end{equation}
\item[{\rm 7.}] The following are equivalent:

{\rm (i)}\quad \; $K_u$ is $|\cdot|_u$-closed;

{\rm (ii)}\quad \, $K_u$ is lineally closed;

{\rm (iii)}\quad  $K_u$ is  Archimedean.

In this case, $|x|_u\in\mathcal M_u(x)$ $($that is $\mathcal M_u(x)=[|x|_u;\infty))$ and $\,B_u[0,1]=[-u;u]_o.$
 \end{enumerate}\end{prop}
 \begin{proof} 1.\; If $v\sim u,$ then $v\in X_u$ and $u\in X_v$ which imply $X_u=X_v$ and $K_u=K_v.$
We have
 \begin{equation}\label{eq1.p1.Ku-equiv}
 \forall \alpha > |v|_u,\quad -\alpha u\le v\le \alpha u. \end{equation}

  Let $x\in X_u$. If $\beta >0$ is such that
  \begin{equation}\label{eq2.p1.Ku-equiv}
 -\beta v\le x\le \beta v, \end{equation}
then
$$
  \forall \alpha > |v|_u,\quad -\alpha\beta  u \le x\le \alpha \beta u.$$

 It follows
 $$
 |x|_u\le\alpha \beta,
 $$
 for all $\beta>0$ for which \eqref{eq2.p1.Ku-equiv} is satisfied and for $\alpha >|v|_u,$  implying
 $|x|_u\le |v|_u |x|_v.$ The second inequality in \eqref{eq.p1.Ku-equiv} follows by symmetry.

  2.\; It  is known that the Minkowski functional corresponding to an absorbing absolutely convex subset $Z$ of a linear space $X$ is a norm iff the set $Z$ is radially bounded in $X$ (i.e. any ray from 0 intersects $Z$ in a bounded interval). Since a cone is almost Archimedean iff any order interval is lineally bounded (Proposition \ref{p.char.a-Arch}), the equivalence  follows.

 3.\; If $0\le x\le y,$ then $\mathcal M_u(y)\subset \mathcal M_u(x)$ and so $|y|_u=\inf\mathcal M_u(y)\ge \inf \mathcal M_u(x)=|x|_u.$

 4.\; The fact that $K_u$ is generating follows from definitions. The normality follows from the fact that the seminorm $|\cdot|_u$ is monotone and Theorem \ref{t4.char-normal-cone}.

 5.\; If $p$ is a seminorm corresponding to an absorbing absolutely convex subset $Z$ of a vector space $X$, then
 $$
 B_p(0,1)\subset Z\subset B_p[0,1],$$
 which in our case yield
 $$
 B_u(0,1)\subset [-u;u]_o\subset B_u[0,1],$$
 which, in their turn, imply the inclusions from 4.

 6.\; We shall prove the inclusions
\begin{equation}\label{eq2.p1.Ku}
 \inter_u(K_u)\subset\ainter(K_u)\subset  K(u)\subset \inter_u(K_u) .\end{equation}

 The first inclusion from above is a general property in topological vector spaces.\vspace{2mm}

\emph{ The inclusion} $\, \ainter(K_u)\subset  K(u).$

 For $x\in \ainter(K_u)$ we have to prove the existence of $\alpha,\beta >0$ such that
 $$
 \alpha u\le x\le \beta u.$$

 Since $x\in \ainter(K_u)$ there exists $\alpha >0$ such that $x+tu\in K_u$ for all $t\in[-\alpha,\alpha]$ which implies
 $x-\alpha u\in K_u,$ that is $x\ge \alpha u.$

 From  \eqref{char.Ku} and the fact that $x\in K_u$ follows the existence of $\beta >0$ such that $x\in\beta[0;u]_o,$ so that $x\le \beta u.$
 \vspace{2mm}

\emph{ The   inclusion} $\, K(u)\subset \inter_u(K_u).$

 If $x\in K(u),$ then there exist $\alpha,\beta>0$ such that $\alpha u\le x\le \beta u.$  But then
 $$
 B_u\big(x,\frac{\alpha}{2}\big)= x+B_u\big(0,\frac{\alpha}{2}\big)\subset x+\frac{\alpha}{2} [-u;u]_o \subset \left[\frac\alpha 2 u;\left(\beta+\frac\alpha 2\right)u\right]_o\subset K_u,
 $$
 proving that $x$ is a $|\cdot|_u$-interior point of $K_u.$

 7.\; The implication (i) $\Rightarrow$ (ii) is a general property.   By Proposition \ref{p.char.Arch}, (ii) $\iff$ (iii).

 It remains to prove the implication (iii) $\Rightarrow$ (i).

 Let $x\in X_u$ be a point in the $|\cdot|_u$-closure of   $K_u.$    Then for every $n\in\mathbb{N}$ there exists $x_n\in K_u$ such that $|x_n-x|_u<\frac 1n.$  By the definition of the seminorm  $\,|\cdot|_u, $
$$
 x_n-x\in  \frac 1n [-u;u]_o,$$
so that
$$
-x\le -x+x_n\le  \frac 1n\,u,$$
for all $n\in\mathbb{N}.$

  By Proposition \ref{p.char.Arch},  this implies $-x\le 0,$ that is $x\ge 0,$   which means that $x\in K_u.$

  Suppose now that the cone $K_u$ is Archimedean. For $x\in X_u\setminus\{0\}$ put  $\alpha:=|x|_u>0.$ Then there exists a sequence $\alpha_n\searrow \alpha$ such that $x\in\alpha_n[-u;u]_o$ for all $n\in\mathbb{N},$  so that
  $$
  \frac{1}{\alpha_n}x+u\ge 0\quad \mbox{and}\quad -\frac{1}{\alpha_n}x+u\ge 0,$$
  for all $n\in \mathbb{N}.$  Since the cone $K_u$ is lineally closed, it follows
 $$
  \frac{1}{\alpha}x+u\ge 0\quad \mbox{and}\quad -\frac{1}{\alpha}x+u\ge 0,$$
which means  $x\in\alpha[-u;u]_o.$

By 4, $[-u;u]_o\subset B_u[0,1].$ If $x\in B[0,1]$ (i.e. $|x|_u\le 1$), then $\mathcal M_u(x)=[|x|_u;\infty),$ and so
$$
x\in |x|_u \,[-u;u]_o\subset  [-u;u]_o.$$
\end{proof}

The above construction corresponds to the one used in locally convex spaces.  For a bounded absolutely convex subset $A$ of a locally convex space $(X,\tau)$ one considers the space $X_A$ generated by $A$,
\begin{equation}\label{def.XA}
X_A=\cup_{\lambda>0}\lambda  A=\cup_{n=1}^\infty n A .\end{equation}

Then $A$ is an absolutely convex absorbing subset of $X_A$ and the attached Minkowski functional
\begin{equation}\label{def.pA}
p_A(x)=\inf\{\lambda >0 : x\in\lambda A\},\;\; x\in A,
\end{equation}
is a norm on $X_A\,.$

\begin{theo} \label{t.B-disc}
  Let $(X,\tau)$ be a Hausdorff locally convex space and $A$ a bounded absolutely convex subset of $X$.
  \begin{enumerate}
    \item[{\rm 1.}] The Minkowski functional $p_A$ is a norm on $X_A$ and the topology generated by $p_A$ is finer than that induced by $\tau $ (or, in other words, the embedding of $(X_A,p_A)$ in $(X,\tau)$ is continuous).
\item[{\rm 2.}] If, in addition,  the set $A$ is  sequentially complete with respect to $\tau,$ then   $(X_A,p_A)$ is a Banach space.  In particular, this is true if the space $X$ is sequentially  complete.
  \end{enumerate}
\end{theo}

In the case when $(X_A,p_A)$ is a Banach space one says that $A$ is a \emph{Banach disc}. These spaces are used to prove that every locally convex space is an inductive limit of Banach spaces and to prove that weakly bounded subsets of a sequentially complete Hausdorff LCS are strongly bounded. (A subset of a $Y$ LCS $X$ is called \emph{strongly bounded} if
$$
\sup\{|x^*(y)| : y\in Y, x^*\in M\}<\infty ,$$
for every weakly bounded subset $M$ of $X^*).$

For details concerning this topic, see the book \cite[\S 3.2]{Bonet}, or \cite[\S 20.11]{Kothe}.

In our case, the normality of $K$ guarantees the completeness of $(X_u,|\cdot|_u).$
\begin{theo}\label{t1.complete-Xu} Let $(X,\tau)$ be a  Hausdorff LCS ordered by a closed normal cone $K$ and $u\in K\setminus\{0\}.$
\begin{enumerate}
  \item[{\rm 1.}] The functional $|\cdot|_u$ is a norm on $X_u$ and the topology generated by $|\cdot|_u$ on $X_u$ is finer than that induced by $\tau$ (or, equivalently, the embedding of $(X_u,|\cdot|_u)$  in $(X,\tau)$ is continuous).
      \item[{\rm 2.}] If  the space $X$ is sequentially complete, then $(X_u,|\cdot|_u)$ is a Banach space.
\item[{\rm 3.}] If $u$ is a unit in $(X,K),$ then $X_u=X.$ If $u\in\inter(K),$ then the topology generated by $|\cdot|_u$ agrees with $\tau.$
\end{enumerate}\end{theo}
\begin{proof} By Theorem  \ref{t2.char-normal-cone},  we can suppose that the topology $\tau$ is generated by a directed family $P$ of $\gamma$-monotone seminorms, for some $\gamma>0.$

1.\; By Proposition \ref{p1.normal-cone-bd},  the set $[-u;u]_o$ is bounded and so $|\cdot|_u$ is a norm. We show that the embedding of $(X_u,|\cdot|_u)$  in $(X,P)$ is continuous.

Let $p\in P.$ The inequalities  $-|x|_u u\le x\le |x|_u u$ imply
$$
 0\le x + |x|_u u\le 2 |x|_u u\quad\mbox{and}\quad  0\le-x + |x|_u u\le 2 |x|_u u,
$$
for all $x\in X_u$

By the $\gamma$-monotonicity of the seminorm $p$ these inequalities imply in their turn
\begin{align*}
  2p(x)\le& p(x + |x|_u u) + p(x - |x|_u u)=  p(x + |x|_u u) + p(-x + |x|_u\, u)\\
  \le& 4\gamma |x|_u\, p(u).
\end{align*}

Consequently, for every    $p\in P,$
\begin{equation}\label{eq1.t1.complete-Xu}
 p(x)\le 2 \gamma p(u) |x|_u ,
 \end{equation}
for all $x\in X_u,$  which shows that the embedding of $(X_u,|\cdot|_u)$ in $(X,\tau)$ is  continuous.

 2.\; Suppose now that $(X,\tau)$ is sequentially complete and let $(x_n)$ be a $|\cdot|_u$-Cauchy sequence in $X_u$.  By \eqref{eq1.t1.complete-Xu}, $(x_n)$ is $p$-Cauchy for every $p\in P,$ so it is $\tau$-convergent to some $x\in X.$ By the Cauchy condition, for every $\varepsilon >0$ there is $n_0\in\mathbb{N}$ such that
 $|x_{n+k}-x_{n}|_u <\varepsilon,$ for all $n\ge n_0$ and all $k\in \mathbb{N}.$ By the definition of the functional $|\cdot|_u,$ it follows
 $$
 -\varepsilon u\le x_{n+k}-x_{n}\le  \varepsilon u,$$
 for all $n\ge n_0$ and all $k\in \mathbb{N}.$ Letting  $k\to\infty,$ one obtains
 $$
 -\varepsilon u\le x-x_{n}\le  \varepsilon u,$$
for all $n\ge n_0$, which implies $x\in X_u$ and $|x-x_n|_u\le \varepsilon, $ for all $n\ge n_0$. This shows that $x_n\xrightarrow{|\cdot|_u} x.$

3.\; If $u$ is a unit in $(X,K)$, then the order interval $[-u;u]_o$ is absorbing, and so $X=\cup_{n=1}^\infty n\,[-u;u]_o=X_u.$

Suppose now that $u\in\inter(K).$ Then $u$ is a unit in $(X,K)$, so that $X=X_u$ and, by 1, the topology $\tau_u$ generated by $|\cdot|_u$ is finer than $\tau,\; \tau\subset \tau_u.$

Since $u\in\inter(K)$, there exists $p\in P$ and $r>0$ such that $B_p[u,r]\subset K.$ Let $x\in X,\, x\ne 0.$

If $p(x)=0$, then $u\pm tx\in B_p[u,r]\subset K$ for every $t>0,$ so that $-t^{-1}u\le x\le t^{-1}u$ for all $t>0,$ which implies $|x|_u=0,$ in contradiction to the fact that $|\cdot|_u$ is a norm on $X$.

Consequently, $p(x)>0$ and $u\pm p(x)^{-1}r x\in B_p[u,r]\subset K,$ that is
$$
-\frac{p(x)}{r}\, u\le x\le  \frac{p(x)}{r}\, u,$$
and so
$$
|x|_u\le \frac{p(x)}{r}.$$

But then, $B_p[0,r]\subset B_{|\cdot|_u}[0,1]$, which implies $B_{|\cdot|_u}[0,1]\in \tau,$ and so $\tau_u\subset \tau.$
 \end{proof}
 \begin{remark}
   Incidentally, the proof of the third assertion of the above theorem gives a proof to Proposition \ref{p1.normal-c-TVS}.
 \end{remark}

 \subsection{The topology of the Thompson metric}\label{Ss.top-T-metric}

 We shall examine some topological properties of  the Thompson extended metric $d$. An extended metric $\rho$ on a set $Z$ defines a topology in the same way as a usual one, via balls. In fact all the properties reduces to the study of metric spaces formed by the components with respect to $\rho$. For instance, a sequence $(z_n)$ in $(Z,\rho)$ converges to some $z\in Z,$
 iff there exists a component $Q$  with respect to $\rho$ and $n_0\in\mathbb{N}$ such that $ z\in Q,\, x_n\in Q$ for $n\ge n_0,$ and $\rho(z_n,z)\to 0$ as $n\to \infty,$ that is $(z_n)_{n\ge n_0}$ converges to $z$ in the metric space $(Q,\rho|_Q).$

The following results are immediate consequences of the definition.
\begin{prop}\label{p1.top-T-metric}
Let $X$ be a vector space ordered by a cone $K$.
\begin{enumerate}
\item[{\rm 1.}] The following inclusions hold
\begin{equation}\label{eq.p1.top-T-metric}
B_d(x,r)\subset[e^{-r}x;e^{r}x]_o\subset B_d[x,r].\end{equation}

If $K$ is Archimedean, then $ B_d[x,r]=[e^{-r}x;e^{r}x]_o.$
\item[{\rm 2.}] If $K$ is Archimedean, then the set $\,[x;\infty)_o:=\{z\in K : x\le z\}\,$ and the order interval $[x;y]_o$ are $d$-closed, for every $x\in K$ and $y\ge x.$
\item[{\rm 3.}]   Let $x,y\in K$ with $x\le y.$ Then the order interval $[x;y]_o$ is $d$-bounded iff $x\sim y.$
\end{enumerate}
\end{prop}\begin{proof} 1.\; If $d(x,y)<r,$ then there exists $s,\, d(x,y)\le s<r,$ such that
$y\in [e^{-s}x;e^{s}x]_o.$ Since $[e^{-s}x;e^{s}x]_o  \subset [e^{-r}x;e^{r}x]_o,$ the first inclusion in \eqref{eq.p1.top-T-metric} follows.
Obviously, $y\in[e^{-r}x;e^{r}x]_o$ implies $d(x,y)\le r.$

Suppose that $K$ is Archimedean and $d(x,y)=r.$ Let $t_n > r$ with $t_n\searrow r.\,$ Then $e^{-t_n}x\le y\le e^{t_n}x$ for all $n.$ Since $K$ is Archimedean, these inequalities imply $e^{-r}x\le y\le e^{r}x.$

2.\; Let $z$ be  in the $d$-closure of $[x;\infty)_o\,.$ Let $t_n>0,\, t_n\searrow 0.$ Then for every $n\in\mathbb{N}$ there exists $z_n\ge x$ such that $d(z,z_n)<t_n,$ implying $x\le z_n\le e^{t_n}z.$  The inequalities $x \le   e^{t_n}z$ yield for $n\to \infty,\, x\le z,$ that is $z\in [x;\infty)_o.$

In a similar way one shows that $[0;x]_0$ is $d$-closed. But then, $[x;y]_o=[x;\infty)_o\cap [0;y]_o$ is also $d$-closed.

3.\; If $[x;y]_o$ is bounded, then $d(x,y)<\infty,$ and so $x\sim y.$ Conversely, if $x\sim y,$ then there exist $\alpha,\beta>0$ such that $\alpha x\le y\le\beta x.$ Then, $x\le z\le y$ implies $x\le z\le y\le\beta x$, and so, by Proposition \ref{p4.T-metric}.1.(iii), $d(x,z)\le\ln \beta.$
\end{proof}

\begin{prop}\label{p2.top-T-metric}
Let $X$ be a vector space ordered by a cone $K$.
\begin{enumerate}
\item[{\rm 1.}] The multiplication by scalars    $\cdot:(0;\infty)\times K\to K$ and the addition $+:K\times K\to K$ are continuous with respect to the Thompson metric.
\item[{\rm 2.}]  If $Q$ is a component of $K$, then the mapping $(\lambda,x,y)\mapsto (1-\lambda)x+\lambda y$ from $[0;1]\times Q^2$ to $Q$  is continuous with respect to the Thompson metric.
\end{enumerate}
\end{prop}\begin{proof}
1.\; Let $(\lambda_0,x_0)\in (0;\infty)\times K$.   Appealing to Proposition \ref{p4.T-metric}.1 it follows
\begin{align*}
  d(\lambda x,\lambda_0x_0)\le& d(\lambda x,\lambda_0x)+d(\lambda_0 x,\lambda_0x_0)\\
  =&|\ln \lambda -\ln\lambda_0| +d(x,x_0) \to 0,
\end{align*}
if $\lambda\to \lambda_0$ and $x\xrightarrow{d}x_0$.

The continuity of the addition can be obtained from  \eqref{eq3.p4.T-metric} (with $\lambda =\mu =1$).

2.\; Let $\lambda,\lambda_0\in[0;1]$ and $x,x_0,y,y_0\in Q.$ This time we shall appeal to the inequalities \eqref{eq3.p4.T-metric} and \eqref{eq3.q-cv} to write
\begin{align*}
  &d((1-\lambda) x+\lambda y,(1-\lambda_0) x_0+\lambda_0 y_0)\le\\&\le  d((1-\lambda) x+\lambda y,(1-\lambda) x_0+\lambda y_0) +
  d((1-\lambda) x_0+\lambda y_0,(1-\lambda_0) x_0+\lambda_0 y_0)\\
  &\le\max\{d(x,x_0),d(y,y_0)\} +\ln\big(|\lambda-\lambda_0|e^{d(x_0,y_0)}+1-|\lambda-\lambda_0|\big)\to 0
\end{align*}
as $\lambda\to \lambda_0,\, x\xrightarrow{d}x_0$ and $y\xrightarrow{d}y_0$.
\end{proof}

\begin{corol}\label{c.T-metric-connex}
Every component of $K$ is path connected with respect to the Thompson metric.
  \end{corol}
  \begin{proof}
    Follows from Proposition \ref{p2.top-T-metric}.2. If $x_0,x_1$ are in the same component, then $\varphi(t)=(1-t)x_0+t x_1,\, t\in [0;1],$ is a path connecting $x_0$ and $x_1.$
  \end{proof}
\begin{remark}\label{re.T-metric-connex}
  The equivalence classes with respect to the equivalence $\sim$ are exactly  the  equivalence classes considered by Jung \cite{jung69} (the equivalence relation considered by Jung is $x \simeq y \iff d(x,y)<\infty,$ see Remark \ref{re.extended-T-metric}). Since these classes are both open and closed, it follows that the components of $K$ with respect to $\sim$ are, in fact, the connected components of $K$ with respect to to the  Thompson (extended) metric $d$.
\end{remark}

  In the following proposition we give a characterization of  $d$-convergent monotone sequences.
  \begin{prop}\label{p.T-converg-seq}
  Let $X$ be a vector space ordered by an Archimedean cone $K$.
\begin{enumerate}
\item[{\rm 1.}]  If $(x_n)$ is an increasing sequence in $K$, then $(x_n)$ is $d$-convergent to an $x\in K$ iff
\begin{equation*}\begin{aligned}
{\rm(i)}\quad &\forall n\in \mathbb{N},\quad x_n\le x,\quad\mbox{and}\\
{\rm(ii)}\quad &\forall \lambda>1,\;\exists k\in \mathbb{N},\quad x\le\lambda x_k.
\end{aligned}\end{equation*}
In this case,  $x=\sup_nx_n$ and there exists $k\in\mathbb{N}$ such that $x_n\in K(x)$ for all $n\ge k.$
\item[{\rm 2.}] If $(x_n)$ is a decreasing sequence in $K$, then $(x_n)$ is $d$-convergent to an $x\in K$ iff
\begin{equation*}\begin{aligned}
{\rm(i)}\quad &\forall n\in \mathbb{N},\quad x\le x_n,\quad\mbox{and}\\
{\rm(ii)}\quad &\forall \lambda\in(0;1),\;\exists k\in \mathbb{N},\quad x\ge\lambda x_k.
\end{aligned}\end{equation*}
In this case,  $x=\inf_nx_n$ and there exists $k\in\mathbb{N}$ such that $x_n\in K(x)$ for all $n\ge k.$
\item[{\rm 3.}] Let $(X,\tau)$ be a TVS ordered by a  cone $K$. If $(x_n)$ is a $d$-Cauchy sequence in $K$ which is $\tau$-convergent to $x\in K$, then $x_n\xrightarrow{d}x.$
    \end{enumerate}
  \end{prop}
  \begin{proof} We shall prove only the assertion 1, the proof of 2 being similar.

  Suppose that the condition (i)  and (ii) hold and let $\varepsilon>0.$ Then $\lambda:=e^\varepsilon>1,$ so that, by  (ii),
  there exists $k\in\mathbb{N}$ such that $x\le \lambda x_k=e^{\varepsilon} x_k.$ Taking into account the monotony of the sequence $(x_n)$ it follows that
  $$e^{-\varepsilon} x_n\le x_n\le x\le e^{\varepsilon} x_n,$$
  for all $n\ge k$, which implies $d(x,x_n)\le\varepsilon$ for all $n\ge k,$ that is $x_n\xrightarrow{d}x$ as $n \to \infty.$

  Conversely, suppose that $(x_n)$ is an increasing sequence in $K$ which is $d$-convergent to $x\in K.$ For $\lambda >1$ put $\varepsilon:=\ln\lambda>0.$ Then there exists $k\in\mathbb{N}$ such that
  $$
  \forall n\ge k,\;\; d(x,x_n)<\varepsilon,$$
  which implies
   $$\forall n\ge k,\;\; e^{-\varepsilon} x_n \le x\le e^{\varepsilon} x_n.$$

 By the second inequality above, $x\le \lambda x_k,$ which shows that (i) holds.  Since $(x_n)$ is increasing the first inequality implies that for every $n\in\mathbb{N}$
 $$ e^{-\varepsilon} x_n \le x ,$$
  for all $\varepsilon >0. $ By Proposition \ref{p.char.Arch}, the cone $K$ is lineally closed, so that the above inequality yields for $\varepsilon\searrow 0,\, x_n\le x$ for all $n\in\mathbb{N},$ that is (i) holds too.

  It is clear that if $x_n\xrightarrow{d}x$, then there exists $k\in\mathbb{N}$ such that $d(x,x_n)\le 1<\infty,$ for all $n\ge k,$ which implies $x_n\in K(x)$ for all $n\ge k.$

  It remains to show that $x=\sup_nx_n.$ Let $y$ be an upper bound for $(x_n)$. Then for every $n\in\mathbb{N},$
  \begin{equation}\label{eq.p.T-converg-seq}
  x_n\le x_{n+k}\le y \iff x_{n+k}\in [x_n;y]_o,
  \end{equation}
  for all $k\in \mathbb{N}.$ By Proposition \ref{p1.top-T-metric}.2 the interval $[x_n;y]_o$ is $d$-closed, so that, letting
$k\to\infty$ in \eqref{eq.p.T-converg-seq} it follows $x\in[x_n;y]_o.$ The inequality $x\le y$ shows that $x=\sup_nx_n.$

3.\; It follows that $(x_n)$ is eventually contained in a component $Q$ of $K,$ so we can suppose $x_n\in Q,\, n\in\mathbb{N}.$ Since $(x_n)$ is $d$-Cauchy, there exists $n_0$ such that $d(x_n,x_{n_0})<1$ for all $n\ge n_0.$ Then $e^{-1}x_{n_0}\le x_n\le e x_{n_0},$  for all $n\ge n_0.$ Letting $n\to \infty,$ one obtains $e^{-1}x_{n_0}\le x\le e x_{n_0},$ which shows that $x\sim x_{n_0},$ that is $x\in Q.$ Now for $\varepsilon >0$ there exists $n_\varepsilon\in\mathbb{N}$ such that $d(x_{n+k},x_{n})<\varepsilon$ for all $n\ge n_\varepsilon$ and all $k\in\mathbb{N}$. Then for every $n\ge n_\varepsilon,\,$ $e^{-\varepsilon}x_{n}\le x_{n+k}\le e^{\varepsilon} x_{n},$  for all $k\in\mathbb{N}.$ Letting $k\to \infty,$ one obtains $e^{-\varepsilon}x_{n}\le x\le e^{\varepsilon} x_{n},$ implying $d(x_n,x)\le\varepsilon$ for all $n\ge n_\varepsilon$, that is  $x_n\xrightarrow{d} x.$
\end{proof}

\subsection{The Thompson metric and order--unit seminorms}

The main aim of this subsection is to show that the Thompson metric and the metric seminorms are equivalent on each component of $K$.
We begin with some inequalities.
\begin{prop}\label{p1.T-metric-u-sn}
Let $X$ be a vector space ordered by a cone $K, u\in K\setminus\{0\}$ and $x,y\in K(u).$ The following relations hold.
 \begin{enumerate}
\item[{\rm 1.}]  \quad $d(x,y)=\ln(\max\{|x|_y, |y|_x\}).$
\item[{\rm 2.}]  \quad
{\rm(i)}\quad $d(x,y)\ge|\ln|x|_u-\ln|y|_u|,$ \; or, equivalently,

 {\rm(ii)}\quad $|x|_u\le e^{d(x,y)}|y|_u\quad\mbox{and}\quad |y|_u\le e^{d(x,y)}|x|_u.$
 \item[{\rm 3.}] \quad $e^{-d(x,u)}\le |x|_u\le e^{d(x,u)}.$
 \item[{\rm 4.}]\quad $|u|_x\le e^{d(x,y)}|u|_y .$
 \item[{\rm 5.}]\quad $d(x,y)\le \ln\big(1+|x-y|_u\cdot\max\{|u|_x,|u|_y\}\big).$
 \item[{\rm 6.}]\quad $\big(e^{d(x,y)}-1\big)\cdot\min\{|u|_x^{-1},|u|_y^{-1}\}\le |x-y|_u\le
 \big(2e^{d(x,y)}+ e^{-d(x,y)}-1\big)\cdot \min\{|x|_u,|y|_u\}.$
 \item[{\rm 7.}]\quad $\big(1-e^{-d(x,u)}\big)\cdot\max\{|u|_x^{-1},|u|_y^{-1}\}\le |x-y|_u.$
\item[{\rm 8.}]\quad $ |x-y|_x\ge 1-e^{-d(x,y)}$ .
   \end{enumerate}
\end{prop} \begin{proof}
1.\; Recalling \eqref{def.Mink}, it is easy to check that
$$
s\in\sigma(x,y) \iff e^s\in\mathcal M_x(y)\cap \mathcal M_y(x),$$
and so
$$
d(x,y)=\ln\big(\inf\big\{\mathcal M_x(y)\cap \mathcal M_y(x)\big\}\big)=\ln(\max\{|x|_y,|y|_x\}).$$

2.\; By  \eqref{eq.p1.Ku-equiv}, $\,|x|_u\le|y|_u|x|_y.$  Taking into account 1, it follows
$$
d(x,y)\ge \ln|x|_y\ge \ln|x|_u-\ln|y|_u.$$

By symmetry, $\,d(x,y) \ge \ln|y|_u-\ln|x|_u,$ so that 2.(i) holds. It is obvious that (i) and (ii) are equivalent.

3.\; Taking $y:= u$ in both the inequalities from 2.(ii), one obtains
$$
|x|_u\le e^{d(x,u)}|y|_u\quad\mbox{and}\quad 1\le  e^{d(x,u)}|x|_u\,.$$

4.\; By  \eqref{eq.p1.Ku-equiv}, $\,|u|_x\le|u|_y|y|_x$ and, by 3, $|y|_x\le e^{d(x,y)},$ hence $|u|_x\le e^{d(x,y)}|u|_y.$

5.\; By   \eqref{eq.p1.Ku-equiv} and the triangle inequality
\begin{align*}
&|y|_x\le |x|_x+|x-y|_x\le 1+|x-y|_u|u|_x\, ,\quad \mbox{and}\\
&|x|_y\le |y|_y+|x-y|_y\le 1+|x-y|_u|u|_y\, ,\end{align*}
so that
$$
\max\{|x|_y,|y|_x\}\le 1+|x-y|_u\cdot\max\{|u|_x,|u|_y\}\, .$$

The conclusion follows from 1.

6.\; The inequality 6 can be rewritten as $|x-y|_u\max\{|u|_x,|u|_y\}\ge e^{d(x,y)}-1,$ so that
$$
|x-y|_u\ge \big(e^{d(x,y)}-1\big)\big[\max\{|u|_x,|u|_y\}\big]^{-1}=\big(e^{d(x,y)}-1\big)\cdot \min\{|u|_x^{-1},|u|_y^{-1}\} .$$

To prove the second inequality, take $s\in\sigma(x,y)$ arbitrary. Then $-(e^s-1)x\le x-y\le (1-e^{-s}s)x,$
so that $0\le x-y+(e^s-1)x\le (e^s-^{-s})x.$ The monotony of $|\cdot|_u$ and the triangle inequality imply
$$
 |x-y|_u-(e^s-1)|x_u\le |x-y-(e^s-1) (e^s-e^{-s})x|_u\le (e^s-e^{-s})|x|_u,$$
 so that $\, |x-y|_u\le \big(2e^s+e^{-s}-1\big)|x|_u. $ Since this holds for every $s\in \sigma(x,y)$ it follows
$$
 |x-y|_u\le \big(2e^{d(x,y)}+ e^{-d(x,y)}-1\big) |x|_u .$$

 By interchanging the roles of $x$ and $y$ in the above inequality, one obtains
 $$
 |x-y|_u\le \big(2e^{d(x,y)}+ e^{-d(x,y)}-1\big) |y|_u .$$

 These two inequalities  imply the second inequality in 6.

7.\; By 4,
$$
|u|_x^{-1}\ge e^{-d(x,y)}|u|_y^{-1}\quad\mbox{and}\quad  |u|_y^{-1}\ge e^{-d(x,y)}|u|_x^{-1},$$
so that
$$
\min\big\{|u|_x^{-1},|u|_y^{-1}\big\}\ge e^{-d(x,y)}\max\big\{|u|_x^{-1},|u|_y^{-1}\big\}.$$

The conclusion follows by 6.

8.\; This can be obtained by taking $u:=x$ in 7.
 \end{proof}

\begin{theo}\label{t1.T-metric-u-sn}
Let $X$ be a vector space ordered by a cone $K$ and $ u\in K\setminus\{0\}$. Then the  Thompson metric and the $u$-seminorm are topologically equivalent on $K(u).$
  \end{theo}\begin{proof}
  We have to show that $d$ and $|\cdot|_u$ have the same convergent sequences, that is
    $$
    x_n\xrightarrow{d}x \iff x_n\xrightarrow{|\cdot|_u} x ,
    $$
for any sequence $(x_n)$ in $K(u)$ and any $x\in K(u)$. But, by Proposition  \ref{p1.Ku}.1,
$$  x_n\xrightarrow{|\cdot|_u} x\iff x_n\xrightarrow{|\cdot|_x} x,$$  hence we have to prove the equivalence
\begin{equation}\label{eq.t1.T-metric-u-sn}
    x_n\xrightarrow{d}x \iff x_n\xrightarrow{|\cdot|_x} x .
   \end{equation}

Suppose that $x_n\xrightarrow{d}x .$ By Proposition    \ref{p1.T-metric-u-sn}.6
$$
|x_n-x|_x\le  2e^{d(x_n,x)}+ e^{-d(x_n,x)}-1\to 0\quad\mbox{as}\quad n\to \infty,$$
showing that  $x_n\xrightarrow{|\cdot|_x} x .$

Conversely, if $x_n\xrightarrow{|\cdot|_x} x, $ then by Proposition    \ref{p1.T-metric-u-sn}.8 ,
$$
 |x_n-x|_x\ge 1-e^{-d(x_n,x)},$$
 which implies $d(x_n,x)\to 0.$
    \end{proof}

    \begin{remark}\label{re.T-metric-u-semin}
    The seminorm $|\cdot|_u$ and the metric $d$ are not metrically equivalent on $X_u.$ Take, for instance,
    $U:=[0;u]_o\cap K(u).$ Then $|x|_u\le 1$ for every $x\in U$. But $U$ is not $d$-bounded because $e^{-n}u$ belongs to $U$ for all $n\in \mathbb{N},$ and $d(x_n,u)=n\to\infty$ for $n\to \infty.$
    \end{remark}
    \begin{corol}[\cite{chen93} or \cite{Hy-Is-Ras}]\label{c.T-metric-tau}
Let $K$ be a   solid normal cone in a Hausdorff LCS  $(X,\tau).$ Then the topology generated by $d$ on $\inter K$ agrees with the restriction of $\tau $ to $\inter K$.
\end{corol}
\begin{proof}
Let $u\in \inter K.$ By Theorem \ref{t1.complete-Xu}, $X_u=X$  and the topology generated by $|\cdot|_u$ agrees with $\tau,$ that is $|\cdot|_u$ is a norm on $X$ generating the topology $\tau.$ Since $K(u)=\inter K,$ Theorem \ref{t1.T-metric-u-sn} implies that $d$ and $|\cdot|_u$ are topologically equivalent on $K(u).$
  \end{proof}

\section{Completeness properties}\label{S.Completeness}
\subsection{Self-bounded sequences and self-complete sets in a cone}

Let $X$ be a vector space ordered   by a cone $K$. A sequence $(x_n)$ in $K$ is called:

\textbullet\quad \emph{self order--bounded from above}  (or \emph{upper self-bounded})  if for every $\lambda >1$ there exists $k\in\mathbb{N}$ such that $x_n\le\lambda x_k$ for all $n\ge k$.

\textbullet\quad \emph{self order--bounded  from below} (or \emph{lower self-bounded})  if for every $\mu\in(0;1)$ there exists $k\in\mathbb{N}$ such that $x_n\ge\mu x_k$ for all $n\ge k$.

\textbullet\quad \emph{self order--bounded}  (or, simply,   \emph{self-bounded})  if  is self order--bounded both from below  and from above.

\begin{remark}
 If the sequence $(x_n)$ is increasing, then it is self order--bounded from above iff for every $\lambda>1$ there exists $k\in\mathbb{N}$ such that $\lambda x_k$ is an upper bound for the sequence $(x_n)$.

  Similarly, if the sequence $(x_n)$ is decreasing, then it is self order--bounded from below iff for every $\mu\in(0;1)$ there exists $k\in\mathbb{N}$ such that $\mu x_k$ is an lower bound for the sequence $(x_n)$.
\end{remark}

The following propositions put in evidence some connections between self order bounded sequences and $d$-Cauchy sequences.
\begin{prop}\label{p1.s-bd-seq}
Let $X$ be a vector space ordered   by a cone $K$.
\begin{enumerate}
 \item[{\rm 1.}] Any $d$-Cauchy sequence in $K$ is self-bounded.
\item[{\rm 2.}] An increasing sequence in $K$ is upper self-bounded iff it is $d$-Cauchy.
\item[{\rm 3.}] A decreasing sequence in $K$ is lower self-bounded iff it is $d$-Cauchy.
\end{enumerate}
\end{prop}\begin{proof}
1.\; Let $(x_n)$ be a $d$-Cauchy sequence in $K$. If $\lambda>1,$ then for   $\varepsilon:=\ln\lambda>0$ there exists $k\in \mathbb{N}$ such that
$$e^{-\varepsilon}x_n\le x_k\ge e^{\varepsilon}x_n\iff \lambda^{-1}x_n\le x_k\ge \lambda x_n\,, $$
for all $n\ge k.$
Consequently   $x_n\le \lambda x_k, $ for all $n\ge k,$ proving that $(x_n)$ is upper self--bounded.

The lower self-boundedness of $(x_n)$ is proved similarly, taking $\varepsilon'=-\ln \mu,$ for $\mu\in(0;1).$

2. It suffices to prove that an increasing upper self-bounded sequence is $d$-Cauchy. For $\varepsilon >0$ let $\lambda=e^\varepsilon$ and $k\in\mathbb{N}$ such that $x_n\le\lambda x_k$ for all $n\ge k.$ It follows that
$$
e^{-\varepsilon}x_{m}\le x_m\le x_n\le \lambda x_k\le\lambda x_m=e^{-\varepsilon}x_m\, ,$$
 for all $n\ge m\ge k$.
Consequently, $d(x_n,x_m)\le \varepsilon,$ for all $n\ge m\ge k$, which shows that the sequence $(x_n)$ is $d$-Cauchy.

The proof of 3 is similar to the proof of 2, so we omit it.
\end{proof}

\begin{prop}\label{p2.s-bd-seq}
Let $X$ be a vector space ordered   by an Archimedean cone $K$ and $(x_n)$ an increasing sequence in $K$. The following statements are equivalent.
\begin{enumerate}
 \item[{\rm 1.}] The sequence $(x_n)$ is $d$-convergent.
\item[{\rm 2.}] The sequence $(x_n)$ is $d$-Cauchy and has a supremum.
\item[{\rm 3.}]The sequence $(x_n)$ is upper self-bounded and has a supremum.
\end{enumerate}

In the affirmative case $x_n\xrightarrow{d}\sup_nx_n.$
\end{prop}\begin{proof}
  1\;$\Rightarrow$\;2\; Follows from Proposition \ref{p.T-converg-seq}.

 2\;$\Rightarrow$\;3.\; Follows from Proposition \ref{p1.s-bd-seq}.1

  3\;$\Rightarrow$\;1.\; If  $x=\sup_nx_n,$ then
  $x_n\le x$ for all $n\in\mathbb{N},$ showing that condition (i) from   Proposition \ref{p.T-converg-seq}.1 holds.
  Now let  $\lambda>1.$  Since $(x_n)$ is upper self-bounded there exists $k\in\mathbb{N}$ such that
  $x_n\le \lambda x_k$ for all $n\ge k,$ and so $x_n\le \lambda x_k$ for all $n\in\mathbb{N}$ (because $(x_n)$ is increasing).  But then $x=\sup_nx_n\le \lambda x_k,$ which shows that condition (ii) of  the same proposition is also fulfilled.
  Consequently $x_n\xrightarrow{d}x.$

  The last assertion follows by the same proposition.
  \end{proof}

 Similar equivalences, with similar proofs, hold for decreasing sequences.
 \begin{prop}\label{p3.s-bd-seq}
Let $X$ be a vector space ordered   by an Archimedean cone $K$ and $(x_n)$ a decreasing sequence in $K$. The following statements are equivalent.
\begin{enumerate}
 \item[{\rm 1.}] The sequence $(x_n)$ is $d$-convergent.
\item[{\rm 2.}] The sequence $(x_n)$ is $d$-Cauchy and has  an infimum.
\item[{\rm 3.}]The sequence $(x_n)$ is lower self-bounded and has  an infimum.
\end{enumerate}

In the affirmative case $x_n\xrightarrow{d}\inf_nx_n.$
\end{prop}

The following proposition emphasizes a kind of duality between upper and lower self-bounded sequences. If   $Y$ is a subset of an ordered set $X, $ then one denotes by $\sup\!|_Y A$ ($\inf\!|_Y A$)    the supremum (resp. infimum) in $Y$ of a subset $A$ of $Y$.  This may differ from the supremum (resp. infimum) of the set $A$ in $X$.

\begin{prop}\label{p4.s-bd-seq}
Let $X$ be a vector space ordered   by a  cone $K,$   $(x_n)$ an increasing, upper self-bounded sequence in $K$, and $(t_k)$ a decreasing sequence of real numbers, convergent to 1. Then there exists a subsequence $(x_{n_k})$ of $(x_n)$ such that the following conditions  are satisfied.
\begin{enumerate}
 \item[{\rm 1.}] The sequence $(y_k)$ given by $y_k=t_kx_{n_k},\, k\in\mathbb{N},\, $ is decreasing and lower self-bounded and
\begin{equation}\label{eq1.p4.s-bd-seq}
\forall n,k\in\mathbb{N},\quad x_n\le y_k.\end{equation}
\item[{\rm 2.}] If the cone $K$ is Archimedean, $x$ is an upper bound for $(x_n)$ and $y$ is a lower bound for $(y_k)$, then $y\le x.$

\item[{\rm 3.}]If the cone $K$ is Archimedean and $(x_n)$ lies in a vector subspace $Y$ of $X$, then the following statements are equivalent.
    \begin{itemize}
      \item[{\rm (a)}] \quad $(x_n)$ has suppremum; \hspace{3cm} {\rm (c)} \quad $(y_k)$ has an infimum;
       \item[{\rm (b)}] \quad $(x_n)$ has suppremum in $Y$;\hspace{2cm}\quad {\rm (d)} \quad $(y_k)$ has an infimum in $Y$;
 \item[{\rm (e)}] \quad there exists $x\in K$ such that
 \begin{equation}\label{eq2.p4.s-bd-seq}
\forall n,k\in\mathbb{N},\quad x_n\le x\le  y_k.\end{equation}
    \end{itemize}
\end{enumerate}

In the affirmative case
$$
\sup\{x_n : n\in \mathbb{N}\}= \sup\!|_Y\{x_n : n\in \mathbb{N}\}=\inf\{y_k : k\in \mathbb{N}\}= \inf\!|_Y\{y_k : k\in \mathbb{N}\}=x,$$
and $x_n\xrightarrow{d}x$ and $y_k\xrightarrow{d}x.$
\end{prop}\begin{proof} 1.\;
  Since $(t_k)$ is decreasing, $\lambda_k:=t_k/t_{k+1}>1,\,k\in \mathbb{N}.$ The upper self-boundedness of the sequence $(x_n)$ implies the existence of $n_1\in\mathbb{N}$ such that
  \begin{equation}\label{eq3.p4.s-bd-seq}
  \forall n\ge n_1,\quad x_n\le\lambda_1x_{n_1}.\end{equation}

  Since $(x_n)$ is increasing, the inequalities \eqref{eq3.p4.s-bd-seq} hold for all $n\in\mathbb{N}$. If $m_2\in \mathbb{N}$ is such that $x_n\le\lambda_2x_{m_2}$ for all $n\in\mathbb{N},$ then $n_2:=1+\max\{n_1,m_2\}>n_1$ and  $x_n\le \lambda_2x_{n_2}$ for all $n\in\mathbb{N}.$ Continuing in this way one obtains a sequence of indices $n_1<n_2 < \dots$ such that
    \begin{equation}\label{eq4.p4.s-bd-seq}
   \quad x_n\le\lambda_kx_{n_k},
\end{equation}
for all    $ n\in \mathbb{N}$ and all  $k\in \mathbb{N}.$

Let $y_k:=t_kx_{n_k},\, k\in\mathbb{N}.$  Putting $n=n_{k+1}$ in \eqref{eq4.p4.s-bd-seq} it follows $y_{k+1}\le y_k.$
By the same inequality
$$
x_n\le t_{k+1}x_n\le t_kx_{n_k}=y_k,$$
for all $n,k\in\mathbb{N}.$

Let  now  $\mu\in(0;1)$. Since $t_k\to 1$ there exists  $k_0$ such that $t_{k_0}<\mu^{-1}.$ But then, by \eqref{eq1.p4.s-bd-seq},
$$
\mu y_{k_0}\le t_{k_0}^{-1}y_{k_0}=x_{n_{k_0}}\le y_k,$$
  for all $k\in\mathbb{N},$ proving that $(y_k)$ is lower self-bounded.

2.\;  Suppose that $x_n\le x,\, n\in\mathbb{N},$ and  $y_k\ge y,\, n\in\mathbb{N}.$
Then, for all $k\in\mathbb{N},$
$$
y\le y_k=t_kx_{n_k}\le t_k x\, .$$

Since $K$ is Archimedean and $t_k\to 1,$ the inequalities $y\le t_kx,\, k\in\mathbb{N},$ yield for $k\to \infty,\, y\le x.$

3.\;  The implications (a)\,$\Rightarrow$\,(b) and (c)\,$\Rightarrow$\,(d) are obvious.

Let as prove   (b)\,$\Rightarrow$\,(d). Observe first that $y_k\in Y,\, k\in\mathbb{N}.$  Let $x=\sup_Y\{x_n : n\in\mathbb{N}\}.$ By \eqref{eq1.p4.s-bd-seq} $y_k$ is an upper bound for $(x_n),$ for every $k\in\mathbb{N}, $  so that $\,x\le y_k\,,$ for all $k\in\mathbb{N}.$  If $y\in Y$ is such that $y\le y_k$ for all
$k\in\mathbb{N},$ then, by 2, $y\le x,$ proving that $x=\inf_Y\{y_k : k\in\mathbb{N}\}.$ On the way we have shown that $x_n\le x\le y_k$, for all $n,k\in\mathbb{N},$ that is the implication (b)\,$\Rightarrow$\,(e) holds too.

Similar reasonings show that (d)\,$\Rightarrow$\,(b), that is (b)\,$\iff$\,(d). The equivalence (a)\,$\iff$\,(c) can be proved in the same way (just let $Y:=X$).

Finally, let us show that (e)\,$\Rightarrow$\,(c). Assume that for some $x\in K,\, x_n\le x\le y_k$ for all $n,k\in\mathbb{N}.$ Suppose that $y\in K$ is such that $y\le y_k $ for all $k\in\mathbb{N}.$ Then, by 2, these inequalities   imply $y\le x,$ showing that $x =\inf_ky_k.$ (Similar arguments show that $x=\sup_ny_n,$ that is (e)\,$\Rightarrow$\,(a)). The equivalence of the assertions from 3 is (over) proven.

The last assertions of the proposition follow from Propositions \ref{p2.s-bd-seq} and \ref{p3.s-bd-seq}.
\end{proof}

Similar results, with similar proofs,   hold for decreasing lower self-bounded sequences.
\begin{prop}\label{p5.s-bd-seq}
Let $X$ be a vector space ordered   by a  cone $K,$   $(x_n)$ a decreasing, lower self-bounded sequence in $K$, and $(t_k)$ an increasing sequence of real numbers, convergent to 1. Then there exists a subsequence $(x_{n_k})$ of $(x_n)$ such that  following conditions  are satisfied.
\begin{enumerate}
 \item[{\rm 1.}] The sequence $(y_k)$ given by $y_k=t_kx_{n_k},\, k\in\mathbb{N},\, $ is increasing and upper self-bounded and
\begin{equation}\label{eq1.p5.s-bd-seq}
\forall n,k\in\mathbb{N},\quad x_n\ge y_k.\end{equation}
\item[{\rm 2.}] If the cone $K$ is Archimedean, $x$ is a  lower bound for $(x_n)$ and $y$ is an upper bound for $(y_k)$, then $y\ge x.$

\item[{\rm 3.}]If the cone $K$ is Archimedean and $(x_n)$ lies in a vector subspace $Y$ of $X$, then the following statements are equivalent.
    \begin{itemize}
      \item[{\rm (a)}] \quad $(x_n)$ has an infimum; \hspace{3cm} {\rm (c)} \quad $(y_k)$ has a supremum;
       \item[{\rm (b)}] \quad $(x_n)$ has an infimum in $Y$;\hspace{2cm}\quad {\rm (d)} \quad $(y_k)$ has a supremum in $Y$;
 \item[{\rm (e)}] \quad there exists $x\in K$ such that
 \begin{equation}\label{eq2.p5.s-bd-seq}
\forall n,k\in\mathbb{N},\quad x_n\ge x\ge  y_k.\end{equation}
    \end{itemize}
\end{enumerate}

In the affirmative case
$$
\inf\{x_n : n\in \mathbb{N}\}= \inf\!|_Y\{x_n : n\in \mathbb{N}\}=\sup\{y_k : k\in \mathbb{N}\}= \sup\!|_Y\{y_k : k\in \mathbb{N}\}=x,$$
and $x_n\xrightarrow{d}x$ and $y_k\xrightarrow{d}x.$
\end{prop}

The following notions will play a crucial role in the study of completeness of the Thompson metric.

Let $X$ be a vector space ordered   by a cone $K$. A nonempty subset $U$ of $K$ is called:

\textbullet\quad \emph{self order--complete from above}  (or \emph{upper self-complete})  if every increasing self-bounded sequence $(x_n)$ in $U$ has a supremum and $\sup_n x_n \in U.$

\textbullet\quad \emph{self order--complete from below} (or \emph{lower self-complete})  if every decreasing self-bounded sequence $(x_n)$ in $U$ has an infimum and $\inf_nx_n \in U.$

\textbullet\quad \emph{self order--complete}  (or, simply,   \emph{self-complete})  if  it is self order--complete both from below  and from above.

If we do not require the supremum (resp. infimum) to be in $U$, then we say that $U$ is\emph{ quasi upper} (resp. \emph{lower}) \emph{self-complete}.

The duality results given in Propositions \ref{p4.s-bd-seq} and \ref{p5.s-bd-seq} have the following important consequence.
\begin{theo}\label{t1.s-complete}
Let $X$ be a vector space ordered   by an Archimedean  cone $K.$ If $U$ is an order--convex, strictly positively-homogeneous, nonempty subset of $K$, then all  six completeness properties given in the above definitions are equivalent.
\end{theo}\begin{proof}
  It is a simple observation that the stated equivalences hold if we show that self-completeness is implied by
   each of the conditions of quasi upper self-completeness and quasi lower self-completeness.

Assume  that $U$ is quasi upper self-complete and show first that $U$ is   upper self-complete.

Let  $(x_n)$ be an increasing upper self-bounded sequence   in $U.$
By hypothesis, there exists $x:=\sup_nx_n\in K.$ Also there exists $k\in\mathbb{N}$ such that $x_n\le 2 x_k$ for all $n\in\mathbb{N}.$
Consequently $x_k\le x\le 2x_k.$ Since $x_k$ and $2x_k$ belong to $U$ and $U$ is order--convex, $x\in U$, proving that $U$ is upper self-complete.

Let us show now that $U$ is   lower  self-complete too. Suppose    that $(x_n)$ is a decreasing lower self-bounded sequence   $(x_n)$ in $U$ and let $(t_k)$ be an increasing sequence of positive numbers which converges to 1 (e.g., $t_k=1-\frac{1}{2k})$. By Proposition  \ref{p5.s-bd-seq} there exists a subsequence $(x_{n_k})$ of $(x_n)$ such that the sequence $y_k:=t_kx_{n_k},\, k\in\mathbb{N},\,$ is increasing and upper self-bounded. Since we have shown that $U$ is   upper self-complete, there exists $x:=\sup_ky_k\in U.$  By the last part of the same proposition, $\inf_nx_n=x\in U,$ proving that $U$ is   lower self-complete.

When $U$  is quasi lower self-complete, the proof that $U$ is self-complete follows the same steps as
before,   using Proposition \ref{p4.s-bd-seq}  instead of Proposition \ref{p5.s-bd-seq}.
\end{proof}

The following corollary shows that we can restrict to order--convex subspaces of $X$.

\begin{corol}\label{c1.s-complete}
Let $X$ be a vector space ordered   by an Archimedean  cone $K$ and $Y$ an order--convex vector subspace of $X$.  If $U$ is an order--convex, strictly positively-homogeneous, nonempty subset of $Y\cap K\,$,
then $U$ is self-complete in X iff $U$ is self-complete in $Y$.
\end{corol}

For a lineally solid cone $K,$ the self-completeness is equivalent to the self-completeness of its algebraic interior.
\begin{prop}\label{p6.s-complete}
Let $X$ be a vector space ordered   by an Archimedean  cone $K$.
\begin{enumerate}
  \item[{\rm 1.}] The cone $K$ is self-complete iff every component  of $K$ is self-complete.
  \item[{\rm 2.}] If, in addition,  $K$ is   lineally solid and $\ainter(K)$ is  self-complete then $K$ is self-complete.
\end{enumerate}
  \end{prop}\begin{proof}
  1.\; Suppose that $K$ is self-complete. Then any component  $Q$   of $K$ is quasi upper self-complete. By Proposition \ref{p1.equiv-cone}, $Q$ satisfies the hypotheses of Theorem \ref{t1.s-complete}, so that it is self-complete.

  Conversely, suppose that every component  of $K$ is self-complete and let $(x_n)$ be an increasing upper self-bounded sequence in $K$. By Proposition \ref{p1.s-bd-seq} the sequence $(x_n)$ is $d$-Cauchy, so there exists $k\in\mathbb{N}$ such that  $d(x_k,x_n)\le 1 <\infty, $ for all $n\ge k,$ implying that the set $\{x_n : n\ge k\}$ is contained in a component $Q$ of $K$. By the self-completeness of $Q$ there exists $x:=\sup\{x_n : n\ge k\}\in Q.$ Since the sequence $(x_n)$ is increasing it follows $x=\sup_nx_n.$ Consequently, $K$ is upper self-complete and, by Theorem \ref{t1.s-complete}, self-complete.

  2.\; Let $(x_n)$ be an increasing, upper self-bounded  sequence in $K$. Fix $x\in \ainter(K).$ Then, by Remark \ref{re.a-int}, the sequence  $y_n:=x_n+x,\, n\in\mathbb{N},$ is contained in $\ainter(K)$ and it is obviously increasing and upper self-bounded. Consequently, $(y_n)$ has a supremum, $y:=\sup_ny_n\in \ainter(K).$  But then there exists  $\sup_nx_n=y-x.$ Therefore the cone $K$ is upper self-complete and, by Theorem \ref{t1.s-complete}, it si self-complete.
  \end{proof}

  \begin{remark}\label{re.s-complete} All the results proven so far can be restated into local versions, by replacing $X $ with     $X_u$, hence $K$ with $K_u$ (where $u\in   K \setminus\{0\}$). In this way, we can weaken the Archimedean condition by requiring only that $K_u$ is Archimedean. In this case, the conditions ``has a supremum", respectively ``has an infimum" must be understood with respect to $X_u$. Consequently, a subset $U$ of $K_u$ can be self-complete in $X_u$, but may  be not self-complete in $X$ (yet, by Corollary \ref{c1.s-complete}, this cannot happen when $K$ is Archimedean).
 Note that the definition of the Thompson metric is not affected by this change
(see Remark \ref{re2.T-metric}).
  \end{remark}

  \subsection{Properties of monotone sequences with respect to order--unit seminorms}\label{Ss.seq-u-semin}
  In this subsection we shall examine the behavior of monotone sequences with respect to to order--unit seminorms, considered in Subsection \ref{Ss.u-semin}. The results are analogous to those established in Subsection \ref{Ss.top-T-metric} for the Thompson metric.

  Throughout this subsection $X$ will be a vector space ordered by a cone $K,\, u\in K\setminus\{0\}$, and $X_u,K_u$ are as in Subsection \ref{Ss.u-semin}. We shall assume also that the cone $K_u=X_u\cap K$ is Archimedean.
  \begin{prop}\label{p1.seq-u-semin}
  Let $(x_n)$ be an increasing sequence in $X_u$ and $x\in X_u\,.$
    \begin{enumerate}
    \item[{\rm 1.}] The sequence $(x_n)$ is $|\cdot|_u$-Cauchy iff
    \begin{equation}\label{eq1.p1.seq-u-semin}
    \forall \varepsilon>0,\; \exists k\in\mathbb{N},\;\forall n\in\mathbb{N},\quad  x_n\le x_k+\varepsilon u.
    \end{equation}
 \item[{\rm 2.}] The sequence $(x_n)$ is $|\cdot|_u$-convergent to $x\in X_u$ iff
  \begin{align*}
    {\rm(i)}\quad &\forall n\in\mathbb{N},\quad x_n\le x;\\
     {\rm(ii)}\quad &\forall \varepsilon>0,\; \exists k\in\mathbb{N},\quad\mbox{such that}\quad  x\le x_k+\varepsilon u.
   \end{align*}
   In the affirmative case, $x=\sup\!|_{X_u}\{x_n :n\in\mathbb{N}\}.$ If $K$ is Archimedean, then $x=\sup_nx_n.$
 \item[{\rm 3.}] The sequence $(x_n)$ is $|\cdot|_u$-convergent to $x\in X_u$ iff   it is $|\cdot|_u$-Cauchy and has a supremum in $X_u$.
     In the affirmative case  $x_n\xrightarrow{|\cdot|_u}\sup\!|_{X_u}\{x_n :n\in\mathbb{N}\}.$
  \end{enumerate}
    \end{prop}\begin{proof}
    1.\;  The sequence $(x_n)$ is $|\cdot|_u$-Cauchy iff
  \begin{equation}\label{eq2.p1.seq-u-semin}
         \forall \varepsilon>0,\; \exists k\in\mathbb{N},\; \forall n\ge m\ge k,\quad -\varepsilon u\le x_n-x_m\le\varepsilon u.
   \end{equation}

Suppose that $(x_n)$ is $|\cdot|_u$-Cauchy and for $\varepsilon>0$ let $k$ be given by the above condition. Taking $m=k$ in the right inequality, one obtains
$ x_n\le x_k+\varepsilon u$ for all $n\ge k,$ and so for all $n\in \mathbb{N}.$

 Suppose now that $(x_n)$ satisfies \eqref{eq1.p1.seq-u-semin}. For $\varepsilon >0$ let $k$ be chosen according to this condition.  By the monotony of $(x_n),\,  x_n-x_m\ge 0\ge-\varepsilon u,$ for all $n\ge m,$    and so the left inequality in \eqref{eq2.p1.seq-u-semin} is true. By  \eqref{eq1.p1.seq-u-semin} and the monotony of $(x_n)$,
 $$
 x_n\le x_k+\varepsilon u\le x_m+\varepsilon u,
 $$
 for all $m\ge k,$ so that the right inequality in \eqref{eq2.p1.seq-u-semin} holds too.

 2.\; We have  $x_n\xrightarrow{|\cdot|_u}x $ iff
      $$
      \forall \varepsilon>0,\; \exists k\in\mathbb{N},\; \forall n\ge k,\quad -\varepsilon u\le x-x_n\le\varepsilon u.
      $$

      The left one of the above inequalities  implies $x_n\le x+\varepsilon u$ for all $n\ge k,$ and so, by the monotony of $(x_n)$, for all $n\in\mathbb{N}.$  Since $K_u$ is Archimedean, letting $\varepsilon\searrow 0$ it follows $x_n\le x,\, $ for every $n\in\mathbb{N}.$ The right inequality implies $x\le x_k+\varepsilon u.$

Conversely, suppose that (i) and (ii) hold. For $\varepsilon>0$ choose $k$ according to (ii). Then, by the monotony of $(x_n)$,
$$
x\le x_k+\varepsilon u\le x_n+\varepsilon u , $$
and so $x-x_n\le \varepsilon u,$ for all $n\ge k.$  By (i), $x_n\le x\le x+\varepsilon u,$ and so $x-x_n\ge -\varepsilon u$ for all $n\in\mathbb{N}.$  Consequently, $-\varepsilon u\le x-x_n\le \varepsilon u $ for all $n\ge k.$

By Proposition \ref{p1.Ku}.7, the cone $K_u$ is   $|\cdot|_u$-closed, and so, by Proposition \ref{p2.order-tvs},  $x_n\xrightarrow{|\cdot|_u}x $ implies $x=\sup\!|_{X_u}\{x_n : n\in\mathbb{N}\}.$

Suppose now that the cone $K$ is Archimedean and that $y\in X$ is such that $x_n\le y$ for all $n\in\mathbb{N}.$ By (ii) for every $\varepsilon>0$ there exists $k\in\mathbb{N}$ such that $x\le x_k+\varepsilon u,$ hence $x \le y+\varepsilon u.$ Letting $\varepsilon\searrow 0$ one obtains $x\le y,$ which proves that $x=\sup_nx_n.$

The direct implication in 3 follows from 1 and 2. Suppose that $(x_n)$ is $|\cdot|_u$-Cauchy and let $x=\sup\!|_{X_u}\{x_n :n\in\mathbb{N}\}.$ For $\varepsilon >0$ there exists $k\in\mathbb{N}$ such that $x_n\le x_k+\varepsilon u$ for all $n\in\mathbb{N},$ implying
$x\le x_k+\varepsilon u.$ Taking into account 2, it follows $x_n\xrightarrow{|\cdot|_u}x. $
\end{proof}

As  before,     similar results, with similar proofs, hold for decreasing sequences.
\begin{prop}\label{p2.seq-u-semin}
  Let $(x_n)$ be a decreasing sequence in $X_u$ and $x\in X_u\,.$
    \begin{enumerate}
    \item[{\rm 1.}] The sequence $(x_n)$ is $|\cdot|_u$-Cauchy iff
    \begin{equation}\label{eq1.p2.seq-u-semin}
    \forall \varepsilon>0,\; \exists k\in\mathbb{N}, \;\forall n\in\mathbb{N},\quad  x_n\ge x_k-\varepsilon u.
    \end{equation}
 \item[{\rm 2.}] The sequence $(x_n)$ is $|\cdot|_u$-convergent to $x\in X_u$ iff
  \begin{align*}
    {\rm(i)}\quad &\forall n\in\mathbb{N},\quad x_n\ge x;\\
     {\rm(ii)}\quad &\forall \varepsilon>0,\; \exists k\in\mathbb{N},\quad\mbox{such that}\quad  x\ge x_k-\varepsilon u.
   \end{align*}
   In the affirmative case, $x=\inf\!|_{X_u}\{x_n :n\in\mathbb{N}\}.$ If $K$ is Archimedean, then $x=\inf_nx_n.$
 \item[{\rm 3.}] The sequence $(x_n)$ is $|\cdot|_u$-convergent to $x\in X_u$ iff   it is $|\cdot|_u$-Cauchy and has an infimum in $X_u$.
     In the affirmative case  $x_n\xrightarrow{|\cdot|_u}\inf\!|_{X_u}\{x_n :n\in\mathbb{N}\}.$
  \end{enumerate}
    \end{prop}

    Now we consider the connection with self-bounded sequences.
    \begin{prop}\label{p3.seq-u-semin}
       Let $(x_n)$ be a $|\cdot|_u$-Cauchy sequence in $K_u.$
       \begin{enumerate}
\item[{\rm 1.}] If there exists $\alpha>0$ and a subsequence $(x_{n_k})$ of $(x_n)$ such that $x_{n_k}\ge\alpha u,$      then $(x_n)$ is self-bounded.
\item[{\rm 2.}] If $(x_n)$ is increasing and there exists $n_0\in\mathbb{N}$ such that $x_{n_0}\in K(u),$ then then $(x_n)$ is self-bounded.
\item[{\rm 3.}] If $(X,\tau)$ is a TVS ordered by a  cone $K$ and $(x_n)$ is a $|\cdot|_u$-Cauchy sequence in $X_u,\, \tau$-convergent to some $x\in X_u,$ then $x_n\xrightarrow{|\cdot|_u}x.$
       \end{enumerate}
    \end{prop}\begin{proof}
      1.\; For $\lambda >1$ put $\varepsilon:=\alpha(\lambda-1).$ Since the sequence $(x_n)$ is $|\cdot|_u$-Cauchy, there exists
      $k\in\mathbb{N}$ such that $x_n\le x_m+\varepsilon u$ for all $n,m\ge k.$ Taking $m=n_k$ and $n\ge n_k (\ge k)$, it follows
      $$
      x_n\le x_{n_k}+(\lambda-1)\alpha u\le x_{n_k}+(\lambda-1)x_{n_k}=\lambda x_{n_k},$$
      proving that $(x_n)$ is upper self-bounded.

The fact that $(x_n)$ is lower self-bounded, can be proved in a similar way, taking $\varepsilon:=\alpha(1-\mu)$ for $0<\mu<1.$

2.\; If $x_{n_0}$ belongs to the component $K(u)$ of $K$, then $x_{n_0}\sim u,$ so there exists $\alpha>0$ such that $x_{n_0}\ge\alpha u.$  It follows $x_{n}\ge\alpha u,$ for all $n\ge n_0,$ and so the hypotheses of 1 are satisfied.

3.\; For $\varepsilon>0$ let $n_\varepsilon\in\mathbb{N}$  be such that $|x_{n+k}-x_n|_u<\varepsilon$ for all $n\ge n_\varepsilon$  and all $k\in\mathbb{N}.$  By \eqref{eq-Mink-u-sn} $\, \varepsilon\in \mathcal M_u(x_{n+k}-x_n),$  that is $-\varepsilon u\le x_{n+k}-x_n\le\varepsilon u,$ for every $n\ge n_\varepsilon$  and all $k\in\mathbb{N}.$ Letting $k\to \infty,$ one obtains $\, -\varepsilon u\le x-x_n\le \varepsilon u,$  implying $|x_n-x|_u\le \varepsilon, $ for all $n\ge n_\varepsilon.$  This shows that $x_n\xrightarrow{|\cdot|_u}x.$
\end{proof}

\subsection{The completeness results}

The following important result shows that the completeness of the Thompson metric $d$ on $K(u)$ and that
of the $u$-norm on $X_u$ are equivalent when $K_u$ is Archimedean (by Remark \ref{re.T-metric-u-semin}   this result is nontrivial)
and also reduces the completeness to the convergence of the monotone Cauchy sequences. We
also show that the completeness of $d$  on $K(u)$ is equivalent to several order--completeness conditions
in $X_u$.

The notation in the following theorem is that of Subsections \ref{Ss.def-T-metric} and \ref{Ss.u-semin}.

\begin{theo}\label{t1.complete-T-m-u-sn}
  Let $X$ be a vector space ordered by a cone $K$ and  let $u\in K\setminus\{0\}$ be such that $K_u$ is Archimedean. Then the following assertions are equivalent.
  \begin{enumerate}
   \item[{\rm 1.}]  $K(u)$ is $d$-complete.
\item[{\rm 2.}] $K(u)$ is self-complete in $X_u$.
\item[{\rm 3.}] $K_u$ is self-complete in $X_u$.
\item[{\rm 4.}] $(X_u,|\cdot|_u)$ is fundamentally $\sigma$-order complete.
\item[{\rm 5.}] $(X_u,|\cdot|_u)$ is monotonically sequentially complete.
\item[{\rm 6.}] $X_u$ is $|\cdot|_u$-complete.
  \end{enumerate}

If, in addition,  $K$ is Archimedean, then the assertions 2 and 3 can be replaced by the stronger versions:
\begin{enumerate}
 \item[{\rm 2$'$.}] $K(u)$ is self-complete (in $X$).
\item[{\rm 3$'.$}] $K_u$ is self-complete (in $ X$).
     \end{enumerate}
\end{theo}\begin{proof}
  1\;$\Rightarrow$\;2.\; If $(x_n)$ is an increasing, upper self-bounded sequence in $K(u)$, then, by Proposition \ref{p1.s-bd-seq}.2, it is $d$-Cauchy, so that it is $d$-convergent to some $x\in K(u)$, and, by Theorem \ref{t1.T-metric-u-sn}, also $|\cdot|_u$-convergent to $x$.  By Proposition \ref{p1.Ku} the cone $K_u$ is $|\cdot|_u$-closed in $X_u,$ so that, by Proposition \ref{p2.order-tvs}.4, $\, x=\sup_nx_n.$ Consequently, $K(u)$
  is upper self-complete in $X_u$. But then, by Proposition \ref{p1.equiv-cone} and Theorem \ref{t1.s-complete}, self-complete in $X_u$.

2$\iff$3.\;  By \eqref{eq1.p1.Ku}, $K(u)=\ainter(K_u),$ so that, by Proposition \ref{p6.s-complete}, $K_u$ is self-complete iff $K(u)$ is self-complete.

4$\iff$5.\; The implication 5\;$\Rightarrow$\;4 is trivial and 4\;$\Rightarrow$\;5 follows by  Propositions \ref{p1.seq-u-semin}.3 and \ref{p2.seq-u-semin}.3.

2\;$\Rightarrow$\;4.\; Using again the fact that $K(u)=\ainter(K_u),$   it is sufficient to show that every increasing $|\cdot|_u$-Cauchy sequence in $K(u)$ has a supremum in $X_u$. Indeed, by Proposition \ref{p2.o-complete} this is equivalent to the fact that  the space $(X_u,|\cdot|_u)$ is fundamentally $\sigma$-order complete. By Proposition \ref{p1.Ku}.4, the cone $K_u$ is normal and generating, so that, by Proposition \ref{p1.o-complete}.3, it is monotonically sequentially complete.

 But, by Proposition
\ref{p3.seq-u-semin}.2, the sequence $(x_n)$ is self-bounded so it has a supremum in $X_u.$

 5\;$\Rightarrow$\;6.\; Since a Cauchy sequence is convergent if has a convergent subsequence, it is sufficient to show that every   sequence $(x_n)$ in $X_u$ satisfying
\begin{equation}\label{eq1.t1.complete-T-m-u-sn}
\forall n\in\mathbb{N}_0,\quad |x_{n+1}-x_n|_u\le \frac{1}{2^n},
\end{equation}
is convergent in $(X_u,|\cdot|_u)$, where $\mathbb{N}_0=\mathbb{N}\cup\{0\}.$

The inequality \eqref{eq1.t1.complete-T-m-u-sn} implies
\begin{equation}\label{eq2.t1.complete-T-m-u-sn}
 -\frac{1}{2^n} \, u\le x_{n+1}-x_n \le \frac{1}{2^n} \, u,
\end{equation}
for all $n\in\mathbb{N}_0$. Writing \eqref{eq2.t1.complete-T-m-u-sn} for $0,1,\dots,n-1$ and adding the obtained inequalities, one obtains
$$
-\left(2-\frac{1}{2^{n-1}}\right)\le x_n-x_0\le \left(2-\frac{1}{2^{n-1}}\right),$$
for all $n\in\mathbb{N}.$ Putting $y_n:=x_n-x_0+\left(2-\frac{1}{2^{n-1}}\right)\,u$ it follows
$$
0\le y_n\le \left(4-\frac{1}{2^{n-2}}\right)\,u,\quad n\in\mathbb{N},
$$
which proves that $y_n\in K_u$ for all $n\in\mathbb{N}.$  Also from $y_{n+1}-y_n=x_{n+1}-x_n+ \frac{1}{2^{n}}\,u$ and \eqref{eq2.t1.complete-T-m-u-sn}, one obtains
$$
0\le y_{n+1}-y_n\le \frac{1}{2^{n-1}}\,u,$$
implying
$$
0\le y_{n+k}-y_n\le \left(\frac{1}{2^{n-1}}+\frac{1}{2^{n}}+\dots+\frac{1}{2^{n+k-2}}\right)\,u<\frac{1}{2^{n-2}}\,u.$$

It follows that $(y_n)$ is an increasing $|\cdot|_u$-Cauchy sequence  in $K_u,$ hence  it is $|\cdot|_u$-convergent to some $y\in X_u$. But then
$$
x_n=y_n+x_0-\left(2-\frac{1}{2^{n-1}}\right)\,u$$
is $|\cdot|_u$-convergent to $y+x_0-2 u\in X_u.$

6\;$\Rightarrow$\;1.\; Again, to prove the completeness of $(K(u),d)$ it is sufficient to show that every sequence $(x_n)$ in $K(u)$ which satisfies
\begin{equation}\label{eq3.t1.complete-T-m-u-sn}
\forall n\in\mathbb{N}_0,\quad d(x_{n+1},x_n)\le \frac{1}{2^{n}},
\end{equation}
is convergent in  $(K(u),d)$. Let $s_0=d(x_0,u).$ Then, by the triangle inequality and \eqref{eq3.t1.complete-T-m-u-sn} applied successively,
$$
d(x_n,x_0)\le 1+\frac{1}{2}+\dots+\frac{1}{2^{n-1}}<2,$$
so that
\begin{equation*}
d(x_n,u)\le d(x_n,x_0)+ d(x_0,u)<2+s_0,
\end{equation*}
implying
\begin{equation}\label{eq4.t1.complete-T-m-u-sn}
e^{-(2+s_0)}u\le x_n \le e^{2+s_0}u
\end{equation}
for all $n\in\mathbb{N}_0.$ The inequality \eqref{eq3.t1.complete-T-m-u-sn} implies
$$
x_{n+1}\le e^{1/2^n}x_n, $$
and
$$
 x_n\le  e^{1/2^n}x_{n+1},$$
so that, taking into account the second inequality in \eqref{eq4.t1.complete-T-m-u-sn}, one obtains the inequalities
$$
x_{n+1}-x_n\le \big(e^{1/2^n}-1\big)x_n\le \big(e^{1/2^n}-1\big)e^{2+s_0}u,$$
and
$$
x_{n+1}-x_n\ge -\big(e^{1/2^n}-1\big)x_n\ge -\big(e^{1/2^n}-1\big)e^{2+s_0}u,$$
which, in their turn, imply
$$
|x_{n+1}-x_n|_u\le \big(e^{1/2^n}-1\big)e^{2+s_0},$$
for all $n\in\mathbb{N}_0.$ The convergence of the series $\sum_n\big(e^{1/2^n}-1\big)e^{2+s_0}$\,\footnote{Follows from the inequality $e^{1/2^n}-1=\frac{1}{2^n}+\frac{1}{2!}\cdot\frac{1}{2^{2n}}+\dots<\frac{1}{2^n}(1+\frac{1}{2!}+\dots)= \frac{1}{2^n}(e-1)$.}and the above inequalities imply that the sequence $(x_n)$ is $|\cdot|_u$-Cauchy, and so   it is $|\cdot|_u$-convergent to some $x\in X_u$.
By Proposition \ref{p2.order-tvs} the order intervals in $X_u$ are $|\cdot|_u$-closed and, by \eqref{eq4.t1.complete-T-m-u-sn}, $x_n\in \big[e^{-(2+s_0)}u;e^{2+s_0}u\big]_u$ it follows $x\in \big[e^{-(2+s_0)}u;e^{2+s_0}u\big]_u\subset K(u).$ Since, by Theorem \ref{t1.T-metric-u-sn}, $d$ and $|\cdot|_u$ are topologically equivalent on $K(u),$ it follows $x_n\xrightarrow{d}x.$  \end{proof}

Combining  Proposition \ref{p6.s-complete} and Theorem \ref{t1.complete-T-m-u-sn} one obtains the following corollaries.
\begin{corol}\label{c1.complete-T-m-u-sn}
If $K$ is Archimedean, then $d$ is complete iff $K $ is self-complete.
\end{corol}
\begin{corol}\label{c2.complete-T-m-u-sn}
  If $K$ is Archimedean and lineally solid, then the following conditions are equivalent.
  \begin{enumerate}
\item[{\rm 1.}] The Thompson metric $d$ is complete.
\item[{\rm 2.}]  The cone $K$ is self-complete.
\item[{\rm 3.}]  The algebraic interior $\ainter(K)$ of $K$ is self-complete.
\item[{\rm 4.}] The algebraic interior $\ainter(K)$ of $K$ is $d$ -complete.
  \end{enumerate}\end{corol}

\subsection{The completeness of the Thompson metric in LCS}
In this subsection we shall prove the completeness of the Thompson metric $d$ corresponding to a normal cone $K$ in a sequentially complete LCS $X$.  In the case of a Banach space the completeness was proved by  Thompson \cite{thomp63}. In the locally convex case we essentially follow   \cite{Hy-Is-Ras}.

Note that if $(X,\rho)$ is an extended metric space, then the completeness of $X$ means the completeness of every component of $X$. Indeed, if $(x_n)$ is a $d$-Cauchy sequence, then there exits $n_0\in\mathbb{N}$ such that $d(x_n,x_{n_0})\le 1, $ for all $n\ge n_0,$ implying that $x_n\in Q,$ for all $n\ge n_0,$ where $Q$ is the component  of $X$ containing $x_{n_0}.$  Also if $x_n\xrightarrow{d}x,$ then there exists $n_1>n_0$ in $\mathbb{N}$ such that $d(x_n,x)\le 1$ for all $n\ge n_1,$ implying that the limit $x$ also belongs to $Q$.

\begin{theo}\label{t1.complete-T-metric}
 Let $(X,\tau)$ be a locally convex space, $K$ a  sequentially complete closed normal  cone in $X$.    Then each component of $K$ is a complete metric space with respect to the Thompson metric  $d.$
  \end{theo}

  By Theorem \ref{t2.char-normal-cone} one can suppose that the topology $\tau$ is generated by a family $P$ of monotone seminorms.

  We start by a lemma which is an adaptation of Lemma 2.3.ii in \cite{kraus-nuss93}, proved for Banach spaces,  to the locally  convex case.
  \begin{lemma}\label{le.T-normal-cone} Let $(X,\tau)$ be a Hausdorff LCS ordered by a closed normal cone $K$ and $d$ the Thompson metric corresponding to $K$. Supposing that $P$ is a directed family of monotone seminorms generating the topology $\tau,$  then for every $x,y\in K\setminus\{0\}$ and every $p\in P,$ the following inequality holds
  \begin{equation}\label{eq1.le.T-normal-cone}
  p(x-y)\le \big(2e^{d(x,y)}+ e^{-d(x,y)}-1\big)\cdot \min\{p(x),p(y)\}\,.
  \end{equation}
  \end{lemma}\begin{proof} We can suppose $d(x,y)<\infty\,$ (i.e. $x\sim y$).  By Proposition \ref{p2.order-tvs} the cone $K$ is Archimedean, so that, by Proposition \ref{p2.T-metric}, $d(x,y)\in \sigma(x,y).$ Putting $\alpha =e^{d(x,y)},$ it follows
  $$
  \alpha^{-1}x\le y\le \alpha x,$$
  so that $(\alpha-1)x\le x-y\le (1-\alpha^{-1})x,$ and so
  $$
  0\le (x-y)+(\alpha-1)x\le (\alpha-\alpha^{-1})x\, .$$

  Let $p\in P.$ By the monotony of $p$ the above inequalities yield
  $$
  p(x-y)-(\alpha-1)p(x)\le p((x-y)+(\alpha-1)x)\le (\alpha-\alpha^{-1})p(x),$$
  and so
  $$
  p(x-y) \le (2\alpha-\alpha^{-1}-1)p(x).$$

  Interchanging the roles of $x$ and $y$ one obtains,
  $$
  p(x-y) \le (2\alpha-\alpha^{-1}-1)p(y),$$
  showing that \eqref{eq1.le.T-normal-cone} holds.
  \end{proof}
\begin{proof}[Proof of Theorem \ref{t1.complete-T-metric}]
  Let $(x_n)$ be a $d$-Cauchy sequence in a component $Q$ of $K$.

  Observe first that the sequence $(x_n)$ is $\tau$-bounded, that is $p$-bounded for every $p\in P.$

  Indeed, if $n_0\in\mathbb{N}$ is such that $d(x_n,x_{n_0})\le 1, $ for all $n\ge n_0,$ then $x_n\le e^{d(x_n,x_{n_0})}x_{n_0} \le e  x_{n_0},$
  for all $n\ge n_0.$ By the monotony of $p$, it follows $p(x_n)\le ep(x_{n_0})$ for all $n\ge n_0$ and every $p\in P.$
  This fact and  the   inequality \eqref{eq1.le.T-normal-cone}   imply that $(x_n)$ is $p$-Cauchy for every $p\in P$, hence it is $P$-convergent to some $x\in X.$

  If $n_0$ is as above, then the inequalities $e^{-1}x_{n_0}\le x_n\le e x_{n_0},$ valid for  all $n\ge n_{0}$, yield for $n\to \infty,\; e^{-1}x_{n_0}\le x\le e x_{n_0},$ showing that $x\sim x_{n_0},$ that is $x\in Q.$

  Since $(x_n)$ is $d$-Cauchy and $\tau$-convergent to $x$, Proposition \ref{p.T-converg-seq}.3 implies that $x_n\xrightarrow{d}x,$ proving the completeness of $(K,d)$.
  \end{proof}

  \subsection{The case of Banach spaces}
  We have seen in the previous subsection that the normality of a cone $K$ in a sequentially complete LCS $X$ is a sufficient condition for the completeness of $K$ with respect to the Thompson metric. In this subsection we   show that, in the case when  $X$ is a Banach space  ordered by a cone $K$,  the completeness of $d$ implies the   normality of $K$. The proof will be based on the following result.

  \begin{theo}\label{t2.complete-T-m-u-sn}
   Let $(X,\|\cdot\|)$ be a Banach space ordered by a   cone $K$ and $u\in K\setminus\{0\}$. Then the following assertions are equivalent.

   \begin{enumerate}
\item[{\rm 1.}] The Thompson metric $d$ is complete on $K(u)$.
\item[{\rm 2.}] $(X_u,|\cdot|_u)$ is a Banach space.
\item[{\rm 3.}] The embedding of $(X_u,|\cdot|_u)$ into $(X,\|\cdot\|)$ is continuous.
\item[{\rm 4.}] The order interval $[0;x]_o$ is $\|\cdot\|$-bounded for every $x\in K(u).$
\item[{\rm 5.}] The order interval $[0;u]_o$ is $\|\cdot\|$-bounded.
\item[{\rm 6.}] Any sequence $(x_n)$ in $K(u)$ which is $d$-convergent to $x\in K(u)$ is also $\|\cdot\|$-convergent to $x$.
   \end{enumerate}
  \end{theo}\begin{proof}
  The equivalence 1$\iff$2 is in fact the equivalence 1$\iff$6 in Theorem \ref{t1.complete-T-m-u-sn}.

  2\; $\Rightarrow$\;3.\; Since both $(X,\|\cdot\|)$ and $(X_u,|\cdot|_u)$ are Banach spaces, by the closed graph theorem it suffices to show that the embedding mapping $I:X_u\to X, \, I(x)=x,\,$ has closed graph. This means that for every sequence $(x_n)$ in  $X_u,\, x_n\xrightarrow{|\cdot|_u}x$ and $ x_n\xrightarrow{\|\cdot\|}y$ imply $y=x.$   Passing to limit for $n\to \infty$  with respect to $\|\cdot\|$ in the inequalities
 $$
    x_n-x+|x_n-x|_u u\ge 0\quad\mbox{and}\quad   x_n-x+|x_n-x|_u u\ge 0\,,
 $$
 and taking into account the fact that  the cone  $K$ is  $\|\cdot\|$-closed,  one obtains
$$
y-x\ge 0\quad\mbox{and}\quad x-y\ge 0,$$
that is $y=x.$

3\; $\Rightarrow$\;4.\; By the continuity of the embedding, there exists $\gamma>0$ such that $\|x\|\le\gamma |x|_u$ for all $x\in X_u.$ By Proposition \ref{p1.Ku} the norm $ |\cdot|_u$ is monotone, so that $0\le z\le x$ implies
$\|z\|\le\gamma |z|_u \le\gamma |x|_u,$ for all $z\in[0;x]_o.$

The implication 4\; $\Rightarrow$\;5\; is obvious.

5\; $\Rightarrow$\;3.\; Let $\gamma>0$ be such that $\|z\|\le \gamma$ for every $z\in[0;u]_u.$  For $x \ne 0$ in $X_u$, the inequalities $\,-|x|_u u\le x\le |x|_u u,$ imply
$$
\frac{x+|x|_uu}{2 |x|_u}\in [0;u]_o,$$
so that $\|x+|x|_uu\|\le 2 \gamma |x|_u \,.$

Hence,
$$
\|x\|- |x|_u\|u\|\le \|x+|x|_uu\| \le 2 \gamma |x|_u ,$$
and so
$$
\|x\|\le (2 \gamma+\|u\|)|x|_u),$$
for all $x\in X_u$, proving the continuity of the embedding of $(X_u,|\cdot|_u)$ into $(X,\|\cdot\|)$.

3\; $\Rightarrow$\;2.\; Let $(x_n)$ be a $|\cdot|_u$-Cauchy sequence in $X_u$. The continuity of the embedding implies that it is $\|\cdot\|$-Cauchy and so, $\|\cdot\|$-convergent to some $x\in X.$ But then, by Proposition \ref{p3.seq-u-semin}.3, $(x_n)$ is $|\cdot|_u$-convergent to $x.$

The implication 3\;$\Rightarrow$\;6 follows by Theorem \ref{t1.T-metric-u-sn}.

6\;$\Rightarrow$\;3.\; Let $(x_n)$ be a sequence in $X_u$ which is $|\cdot|_u$-convergent to $x\in X_u.$ Then $(x_n)$ is $|\cdot|_u$-bounded, so there exists $\alpha>0$ such that $-\alpha u\le x_n\le\alpha u.$  It follows that
$y_n:=x_n+(\alpha+1)u\in[u;(2\alpha+1)u]_o,$ and so $y_n\in K(u),\, n\in\mathbb{N},$ and $y_n\xrightarrow{|\cdot|_u}x+(\alpha+1)u.$ By Theorem \ref{t1.T-metric-u-sn}, $\,y_n\xrightarrow{d}x+(\alpha+1)u,$ so that, by hypothesis, $y_n\xrightarrow{\|\cdot\|}x+(\alpha+1)u.$ It follows $x_n\xrightarrow{\|\cdot\|}x ,$ proving the continuity of the embedding.
\end{proof}

   Now we present several conditions equivalent to the completeness of the Thompson metric.
  \begin{theo}\label{t2.complete-T-metric-B-sp}
  Let $(X,\|\cdot\|)$ be a Banach space ordered by a   cone $K$. The following assertions are equivalent.
    \begin{enumerate}
\item[{\rm 1.}]  The Thompson metric $d$ is complete.
\item[{\rm 2.}] The cone $K$ is self-complete.
\item[{\rm 3.}] The cone $K$ is normal.
\item[{\rm 4.}] The norm topology on $K$ is weaker than the topology of $d$.
    \end{enumerate}
  \end{theo}\begin{proof}
  The equivalence 1$\iff$2 follows by Corollary \ref{c1.complete-T-m-u-sn} (remind that, by Proposition \ref{p2.order-tvs}, the  cone $K$ is Archimedean).

2$\iff$3.\; By Proposition \ref{p6.s-complete}, the cone $K$ is self-complete iff each component of $K$ is self-complete. By Theorem \ref{t2.complete-T-m-u-sn}, this happens exactly when the order interval $[0;x]_o$ is $\|\cdot\|$-bounded for every $x\in K,$ which is equivalent to the fact that the order intervals $[x;y]_o$ are $\|\cdot\|$-bounded for all $x,y\in K. $ By Theorem \ref{t4.char-normal-cone} this is equivalent to the normality of  $K$.

1$\iff$4.\; By Theorem \ref{t2.complete-T-m-u-sn} the cone $K$ is $d$-complete iff the norm topology on each component of $K$ is weaker that the topology generated by $d$, and this is equivalent to 4.
  \end{proof}

\begin{remark}
  By Theorem \ref{t2.complete-T-metric-B-sp} in the case of an ordered Banach space the normality of the cone is both necessary and sufficient for the completeness of the Thompson metric. The proof, relying on Theorem  \ref{t2.complete-T-m-u-sn}, uses  the closed graph theorem and the fact that a   cone in a Banach space is normal iff every order interval is norm bounded. As these results are not longer true in arbitrary LCS, we ask the following question.

 \textbf{Problem.} \emph{ Characterize  the class of LCS for which the normality of $K$ is also necessary for the completeness of the Thompson metric (or, at least, put in evidence a reasonably large class of such spaces).}
\end{remark}

\providecommand{\bysame}{\leavevmode\hbox to3em{\hrulefill}\thinspace}
\providecommand{\MR}{\relax\ifhmode\unskip\space\fi MR }
\providecommand{\MRhref}[2]{%
  \href{http://www.ams.org/mathscinet-getitem?mr=#1}{#2}
}
\providecommand{\href}[2]{#2}


\begin{thebibliography}{100}


\bibitem{akian-nuss13}
M.~Akian, S.~ Gaubert and  R.~Nussbaum, \emph{Uniqueness of the fixed point of nonexpansive semidifferentiable maps},
 arXiv:1201.1536v2 (2013).

\bibitem{Alipr}
Ch.~D. Aliprantis and R. Tourky, \emph{Cones and Duality}, Graduate
  Studies in Mathematics, vol.~84, American Mathematical Society, Providence,
  RI, 2007.

\bibitem{bauer-bear69}
H.~Bauer and H.~S. Bear, \emph{The part metric in convex sets}, Pacific J. Math. \textbf{30} (1969),
15--33.

\bibitem{bear-weis67}
H.~S. Bear and M.~L.  Weiss, \emph{An intrinsic metric for parts}, Proc. Amer. Math. Soc. \textbf{18}
(1967), 812--817.

\bibitem{birkhoff57}
G.~Birkhoff, \emph{Extensions of Jentzsch's theorem},
Trans. Amer. Math. Soc. \textbf{85} (1957), 219--227.

 \bibitem{BreckW}
W.~W. Breckner, \emph{Rational s-convexity. A generalized Jensen-convexity},
 Cluj University Press, Cluj-Napoca, 2011.

 \bibitem{bushell73a}
 P.~J. Bushell, \emph{Hilbert's projective metric and positive contraction mappings in a
Banach space}, Arch. Ration. Mech. Anal., \textbf{52} (1973), 330--338.

\bibitem{bushell73b}
\bysame, \emph{On the projective contraction ratio for positive linear mappings},
J. London Math. Soc. (2) \textbf{6} (1973), 256--258.

\bibitem{chen93}
Y.-Z. Chen, \emph{Thompson's metric and mixed monotone operators}, J.
  Math. Anal. Appl. \textbf{177} (1993), no.~1, 31--37.

\bibitem{chen99}
\bysame, \emph{A variant of the Meir-Keeler-type theorem in ordered Banach
  spaces}, J. Math. Anal. Appl. \textbf{236} (1999), no.~2, 585--593.

 \bibitem{chen01a}
\bysame, \emph{On the stability of positive fixed points}, Nonlinear Anal.,
  Theory Methods Appl. \textbf{47} (2001), no.~4, 2857--2862.

\bibitem{chen02}
\bysame, \emph{Stability of positive fixed points of nonlinear operators},
  Positivity \textbf{6} (2002), no.~1, 47--57.

\bibitem{Deimling}
K. Deimling, \emph{Nonlinear Functional Analysis}, Springer-Verlag, Berlin,
  1985.

\bibitem{Guo}
D. Guo, Y.~J. Cho, and J. Zhu, \emph{Partial ordering methods in
  nonlinear problems}, Nova Science Publishers Inc., Hauppauge, NY, 2004.

 \bibitem{hat-molnar14}
  O.~Hatori  and L.~Moln\'ar,  \emph{Isometries of the unitary groups and Thompson isometries of the spaces of invertible positive elements in $C^*$-algebras}, J. Math. Anal. Appl. 409 (2014), no. 1, 158--167.

\bibitem{hilbert1895}
D.~Hilbert, \emph{\"Uber die gerade Linie als k\"urzeste Verbindung zweier Punkte}, Math.
Ann., \textbf{46} (1895), 91--96.

\bibitem{Hy-Is-Ras}
D.~H. Hyers, G. Isac, and T.~M. Rassias, \emph{Topics in
  Nonlinear Analysis} \& \emph{Applications}, World Scientific Publishing Co. Inc.,
  River Edge, NJ, 1997.

 \bibitem{iz-nakamura09}
S.~Izumino and N.~Nakamura,  \emph{Geometric means of positive operators}, II. Sci. Math. Jpn. \textbf{69} (2009), no. 1, 35--44.

\bibitem{Jameson}
G. Jameson, \emph{Ordered Linear Spaces}, Lecture Notes in Mathematics,
  Vol. 141, Springer-Verlag, Berlin, 1970.

\bibitem{jung69}
C. F. K. Jung,  \emph{On generalized complete metric spaces}, Bull. Amer. Math. Soc. \textbf{75} (1969),
113--116.

  \bibitem{Kothe}
G.~K\"othe, \emph{Topological Vector Spaces I.}, Springer Verlag,
  Berlin-Heidelberg-New York, 1969.

\bibitem{kraus-nuss93}
U.~Krause and R.~D. Nussbaum,\emph{ A limit set trichotomy for self-mappings of normal
cones in Banach spaces}, Nonlinear Anal. \textbf{20} (1993), 855--870.

\bibitem{Lem-Nuss12}
B.~Lemmens  and R.~D. Nussbaum,  \emph{Nonlinear Perron-Frobenius Theory},
Cambridge Tracts in Mathematics, 189. Cambridge University Press, Cambridge, 2012.

\bibitem{lem-nuss13}
\bysame, \emph{Birkhoff's version of Hilbert's metric and its applications in analysis},
arXiv:1304.7921 (2013).

\bibitem{lins07}
B.~Lins, \emph{A Denjoy-Wolff theorem for Hilbert metric nonexpansive maps on polyhedral domains},
Math. Proc. Camb. Phil. Soc. 143 (2007), 157--164.

\bibitem{lins-nuss06}
B.~Lins and R.~D. Nussbaum, \emph{Iterated linear maps on a cone and Denjoy-Wolff theorems}, Linear Algebra Appl. \textbf{416} (2006), no. 2-3, 615--626.

\bibitem{lins-nuss08}
\bysame, \emph{Denjoy-Wolff theorems, Hilbert metric nonexpansive maps and reproduction-decimation operators}, J. Funct. Anal. \textbf{254} (2008), no. 9, 2365--2386.

\bibitem{molnar09}
 L.~Moln\'ar,  \emph{Thompson isometries of the space of invertible positive operators},
 Proc. Amer. Math. Soc. \textbf{137} (2009), no. 11, 3849--3859.

 \bibitem{nakamura09} 	
N.~Nakamura, \emph{Geometric means of positive operators},
Kyungpook Math. J.  49 (2009), No. 1, 167--181.

\bibitem{ng72}
K.~F. Ng,  \emph{On order and topological completeness}, Math. Ann. \textbf{196} (1972), 171--176.

\bibitem{Nuss88}
 R.~D. Nussbaum, \emph{Hilbert's Projective Metric and Iterated Nonlinear Maps}, Mem. Amer. Math.
Soc. Vol. 75, No. 391 (1988), iv+137 pp.

\bibitem{Nuss89}
\bysame, \emph{Iterated Nonlinear Maps and Hilbert's Projective Metric}, Mem. Amer. Math.
Soc. Vol. 79, No. 401 (1989), iv+118 pp.

\bibitem{nuss07}
\bysame, \emph{Fixed point theorems and Denjoy-Wolff theorems for Hilbert's projective metric in infinite dimensions}, Topol. Methods Nonlinear Anal. \textbf{29} (2007), no. 2, 199--249.

\bibitem{nuss-walsh04}
R.~D. Nussbaum and C.~A.  Walsh,  \emph{A metric inequality for the Thompson and Hilbert
geometries}, JIPAM. J. Inequal. Pure Appl. Math. \textbf{5}, 3 (2004), Article 54, 14 pp. (electronic).

\bibitem{Peressini}
A.~L. Peressini, \emph{Ordered Topological Vector Spaces}, Harper \& Row
  Publishers, New York, 1967.

\bibitem{Bonet}
P.~P\'erez Carreras and J.~Bonet,   \emph{Barrelled Locally Convex Spaces},
North-Holland Mathematics Studies, vol. 131 (Notas de Matem\'atica, vol. 113),
North-Holland, Amsterdam,  1987.



\bibitem{RusM10}
M.-D. Rus,  \emph{The Method of Monotone Iterations for Mixed Monotone Operators},
  Babes-Bolyai University, Ph.D. Thesis, Cluj-Napoca, 2010.

\bibitem{rusm11}
\bysame, \emph{Fixed point theorems for generalized contractions in partially
  ordered metric spaces with semi-monotone metric}, Nonlinear Anal. \textbf{74}
  (2011), no.~5, 1804--1813.

  \bibitem{samelson57}
  H.~Samelson, \emph{On the Perron-Frobenius theorem}, Michigan Math. J., \textbf{4} (1957), 57--59.

\bibitem{Schaef}
H.~H. Schaefer, \emph{Topological Vector Spaces}, Third printing corrected,
Graduate Texts in Mathematics, Vol. 3, Springer-Verlag, New
  York, 1971.

\bibitem{thomp63}
A.~C. Thompson, \emph{On certain contraction mappings in a partially ordered
  vector space.}, Proc. Amer. Math. Soc. \textbf{14} (1963), 438--443.

\bibitem{turinici80}
M.~Turinici,  \emph{ Maximal elements in a class of order complete metric spaces},
Math. Japon. \textbf{25} (1980), no. 5, 511--517.

\bibitem{wong72}
Y.~C. Wong,  \emph{Relationship between order completeness and topological completeness}, Math.
Ann. \textbf{199} (1972), 73–82.

\bibitem{Wong-Ng}
Y.~C. Wong and K.~F. Ng, Kung Fu, \emph{Partially Ordered Topological Vector Spaces},
Oxford Mathematical Monographs, Clarendon Press, Oxford, 1973.

\bibitem{Zali}
C.~Z{\u{a}}linescu, \emph{Convex Analysis in General Vector Spaces}, World
  Scientific Publishing Co. Inc., River Edge, NJ, 2002.

\end{thebibliography}
\end{document}